\documentclass[a4paper,reqno,final]{amsart}

\usepackage{amsmath,amssymb,amsthm}
\usepackage{fullpage,enumitem}
\usepackage{mathtools}
\mathtoolsset{showonlyrefs,showmanualtags}
\usepackage{multirow,array}
\usepackage{showkeys,version}
\usepackage{tikz} 
\usetikzlibrary{cd}
\usepackage{dynkin-diagrams}
\hypersetup{bookmarks=false,draft=false,breaklinks,colorlinks,dvipdfmx}

\numberwithin{equation}{subsection}
\numberwithin{figure}{subsection}
\numberwithin{table}{subsection}
\allowdisplaybreaks
\newenvironment{Ack}%
{\par \vspace{\baselineskip}%
 \noindent \textbf{Acknowledgements.}}%
{\par \vspace{\baselineskip}}

\theoremstyle{definition}
\newtheorem{thm}{Theorem}[subsection]
\newtheorem{prp}[thm]{Proposition}

\newtheorem{lem}[thm]{Lemma}
\newtheorem{dfn}[thm]{Definition}
\newtheorem{fct}[thm]{Fact}
\newtheorem{rmk}[thm]{Remark}
\newtheorem{eg}[thm]{Example}
\newtheorem{mth}{Theorem}

\newcommand{\oa}{\overrightarrow}
\newcommand{\ol}{\overline}

\newcommand{\wt}{\widetilde}
\newcommand{\ep}{\epsilon}

\newcommand{\vp}{\varphi}

\newcommand{\inj}{\hookrightarrow}
\newcommand{\lto}{\longrightarrow}
\newcommand{\lrt}{\leftrightarrow}
\newcommand{\lrto}{\longleftrightarrow}
\newcommand{\lmto}{\longmapsto}
\newcommand{\linj}{\lhook\joinrel\longrightarrow}
\newcommand{\xr}{\xrightarrow}
\newcommand{\sto}{\xr{\sim}}

\newcommand{\tbcup}{\textstyle \bigcup}

\newcommand{\hf}{\frac{1}{2}}
\newcommand{\thf}{\tfrac{1}{2}}
\newcommand{\shf}{1/2}

\newcommand{\RY}{\mathrm{RY}}
\newcommand{\tsp}{\mathrm{sp}}

\newcommand{\bK}{\mathbb{K}}
\newcommand{\bN}{\mathbb{N}}
\newcommand{\bQ}{\mathbb{Q}}
\newcommand{\bR}{\mathbb{R}}
\newcommand{\bZ}{\mathbb{Z}}
\newcommand{\bfb}{\mathbf{b}}
\newcommand{\frS}{\mathfrak{S}}

\DeclareMathOperator{\id}{id}
\DeclareMathOperator{\sh}{sh}
\DeclareMathOperator{\GL}{GL}
\DeclareMathOperator{\hgt}{ht}
\DeclareMathOperator{\wgt}{wt}
\DeclareMathOperator{\dir}{d}

\DeclareMathOperator{\End}{End}

\newcommand{\abs}[1]{\left|{#1}\right|}
\newcommand{\br}[1]{\left< #1 \right>}
\newcommand{\bbr}[1]{\bigl< #1 \bigr>}
\newcommand{\tbr}[1]{\langle #1 \rangle}
\newcommand{\pr}[1]{\left\{ #1 \right\}}
\newcommand{\bpr}[1]{\bigl\{ #1 \bigr\}}
\newcommand{\rst}[2]{\left.#1\right|_{#2}}

\begin{document}

\title{Specializing Koornwinder polynomials to %
 Macdonald polynomials of type $B,C,D$ and $B C$}
\author{Kohei Yamaguchi, Shintarou Yanagida}
\date{May 3, 2021} 
\address{Graduate School of Mathematics, Nagoya University.
Furocho, Chikusaku, Nagoya, Japan, 464-8602}.
\email{d20003j@math.nagoya-u.ac.jp, yanagida@math.nagoya-u.ac.jp}
\thanks{S.Y.\ is supported by supported by JSPS KAKENHI Grant Number 19K03399, %
 and also by the JSPS Bilateral Program %
 ``Elliptic algebras, vertex operators and link invariant".}

\begin{abstract}
We study the specializations of parameters in Koornwinder polynomials to obtain 
Macdonald polynomials associated to the subsystems of the affine root system 
of type $(C_n^\vee,C_n)$ in the sense of Macdonald (2003), and summarize them 
in what we call the specialization table.
As a verification of our argument, we check the specializations to type $B,C$ and $D$
via Ram-Yip type formulas of non-symmetric Koornwinder and Macdonald polynomials. 
\end{abstract}

\maketitle
{\small \tableofcontents}

\section{Introduction}\label{s:intro}

In \cite{Mp}, Macdonald introduced families of multivariate $q$-orthogonal polynomials
associated to various root systems, which are today called the Macdonald polynomials.
Each family has additional $t$ parameters 
corresponding to the Weyl group orbits in the root system.
Following this work, in \cite{K}, Koornwinder introduced a multivariate analogue 
of Askey-Wilson polynomial, having additional five parameters aside from $q$,
which is today called the Koornwinder polynomial.
It was also shown in \cite{K} that by specializing these five parameters, 
we can obtain the Macdonald polynomials 
of type $(B C_n,B_n)$ and $(B C_n,C_n)$ in the sense of \cite{Mp}.
Today, these families of multivariate $q$-orthogonal polynomials 
are called the Macdonald-Koornwinder polynomials \cite{Cb,H,M,St3}.

After the development of the representation theoretic approach 
\cite{C1,C2,C3,C4,C5,N,Sa,Sa2,St,vD} using the (double) affine Hecke algebras, 
there appeared several versions of unified formulation 
of the Macdonald-Koornwinder polynomials \cite{Cb,H,M,St3}.
These studies are now called the Macdonald-Cherednik theory.

The specialization argument given by Koornwinder in \cite{K} 
is now understood in a more general form.
First, after the studies in \cite{N,Sa,Sa2,St,vD}, the Koornwinder polynomial can be 
formulated as the Macdonald polynomial associated to the affine root system 
of type $(C_n^\vee,C_n)$ in the sense of \cite{M}.
See also \cite{St2,St3} for the relevant explanation.
Then, as mentioned in \cite[p.12, (5.17)]{M}, the Macdonald polynomials associated to 
all the subsystems of type $(C_n^\vee,C_n)$ can be obtained by 
specializing the five parameters of the Koornwinder polynomial
in the way respecting the orbits of the extended affine Weyl group
acting on the affine root systems.
See also the comment in \cite[6.19]{H}.

However, it seems that the detailed explanation of the specialization argument 
is not given in literature.
The aim of this note is to clarify this point.

What troubled the authors in the early stages of the study is that there are 
tremendously many notations for the affine root systems and 
the parameters of Macdonald-Koornwinder polynomials,
and that even for the work \cite{Mp} and the book \cite{M} both by Macdonald,
there seems no explicit comparison in literature.
To the authors' best knowledge, in the present writing this note, 
the most general framework of the theory of Macdonald-Koornwinder polynomials
is given by Stokman \cite{St3}, which is based on the approach of Haiman \cite{H}.
It treats uniformly the four classes of Macdonald-Koornwinder polynomials:
$\GL_n$, the untwisted case, the twisted case, and the Koornwinder case.
The formulation by Macdonald in \cite{M} treats the latter three cases
along this classification.

Although it would be the best to work in the framework of \cite{St3},
we gave up to do so due to the following reasons.
First, since we are interested in the specialization of Koornwinder polynomials,
we may ignore $\GL_n$ case, and the formulation of \cite{M} will be enough.
Second, we are also motivated by Ram-Yip type formulas of non-symmetric 
Macdonald-Koornwinder polynomials \cite{RY,OS}, and will check our specialization 
argument in the level of those formulas. 
The calculations in the check are based on the recent paper \cite{YK}
by the second named author, which mainly follows the notation in \cite{M}.
Let us mention that some specialization arguments are 
given in \cite[Example 9.3.28, Remark 9.3.29]{St3}.

After these considerations, we decided to use the notation in the following literature:
\begin{enumerate}[nosep]
\item \label{i:intro:M}
\cite{M} for affine root systems. 
\item \label{i:intro:N}
\cite{N} for the parameters of Koornwinder polynomials.
\end{enumerate}

Let us explain \eqref{i:intro:M} in detail. We use the word ``affine root system" 
in the sense of \cite[\S 1.2]{M}, which originates in \cite{M1}. 
The word ``irreducible finite root system" means an irreducible root system in \cite{M}.
We also denote by $\vee$ the dualizing of finite and affine root systems.
Then, as explained in \cite[\S 1.3]{M}, similarity classes of irreducible 
affine root systems are divided into three cases:
\begin{itemize}[nosep]
\item 
reduced and of the form $S(R)$ with $R$ an irreducible finite root system.
\item 
reduced and of the form $S(R)^\vee$ with $R$ an irreducible finite root system.
\item
non-reduced and of the form $S_1 \cup S_2$ with $S_1$ and $S_2$ reduced affine root systems.
\end{itemize}
The appearing $R$ is one of the types 
$A_n$, $B_n$, $C_n$, $D_n$, $BC_n$, $E_6$, $E_7$, $E_8$, $F_4$ and $G_2$.
According to the type of $R$, we say
\begin{itemize}[nosep]
\item 
$S(R)$ is of type $X$ if $R$ is of type $X$,
\item 
$S(R)^\vee$ is of type $X^\vee$ if $R$ is of type $X$,
\item
a non-reduced system $S_1 \cup S_2$ is of type $(X,Y)$ 
if $S_1$ and $S_2$ are of type $X$ and $Y$, respectively.
\end{itemize}
We refer \cite[(1.3.1)--(1.3.18))]{M} for explicit descriptions of 
these irreducible affine root systems,
although some of them will be displayed in the main text.

As explained in \cite[\S 1.4]{M}, Macdonald developed a unified formulation to 
associate a family of $q$-orthogonal polynomials to each of 
the following pairs $(S,S')$ of irreducible affine root systems.
\begin{enumerate}[nosep, label=(\alph*)]
\item \label{i:intro:r}
$(S,S')=(S(R),S(R^\vee))$ with $R$ an irreducible finite root system.
\item \label{i:intro:rv}
$S=S'=S(R)^\vee$ with $R$ an irreducible finite root system.
\item \label{i:intro:nr}
$S=S'$ is non-reduced of type $(X,Y)$.
\end{enumerate}
For each pair $(S,S')$, we have the associated non-symmetric \cite[\S 5.2]{M}
and symmetric \cite[\S 5.3]{M} Macdonald polynomials.
For the reference in the main text, let us introduce:

\begin{dfn}\label{dfn:type}
We call the non-symmetric and symmetric Macdonald polynomials associated to 
$(S,S')$ in the class \ref{i:intro:r}, \ref{i:intro:rv} and \ref{i:intro:nr}
\emph{the non-symmetric and symmetric Macdonald polynomials of 
type $X$, $X^\vee$ and $(X,Y)$}, respectively.
\end{dfn}

In particular, the Koornwinder polynomial is the Macdonald polynomial 
of type $(C_n^\vee,C_n)$.
The detail of the affine root system of type $(C_n^\vee,C_n)$ 
will be explained in \S \ref{ss:CvC}.

As mentioned before, in \cite[p.12]{M}, Macdonald gives a comment that 
the affine root system of type $(C_n^\vee,C_n)$ 
has as its subsystem all the non-reduced affine root systems and 
the classical affine root systems of type 
$B_n$, $B_n^\vee$, $C_n$, $C_n^\vee$, $BC_n$ and $D_n$. 
Also, at \cite[(5.17)]{M}, he comments that an appropriate specialization of 
parameters in the Koornwinder polynomials yields the Macdonald polynomials 
associated to the corresponding subsystem.
Let us state again that the aim of this note is to clarify this point.

Let us turn to the explanation of \eqref{i:intro:N}, 
the notation of the five parameters of the Koornwinder polynomial. 
We will use those introduced by Noumi in \cite{N}:
\begin{align}\label{eq:intro:N}
  t,t_0,t_n,u_0,u_n.
\end{align}
Let us call them \emph{the Noumi parameters} for distinction. 
The details will be explained in \S \ref{ss:HCC}.


Now we can explain the main result of this note.

\begin{mth}[{Propositions \ref{prp:C}, \ref{prp:B}--\ref{prp:BvB}}]\label{thm:1}
For each type $X$ listed in Table \ref{tab:sp} and 
for each (not necessarily) dominant weight $\mu$ of type $C_n$, 
the specialization of the Noumi parameters 
in the (non-symmetric) Koornwinder polynomial with weight $\mu$
yields the (non-symmetric) Macdonald polynomial with $\mu$ of type $X$
in the sense of Definition \ref{dfn:type}.

\begin{table}[htbp]
\centering
\begin{tabular}{ll||lllll|ll||lllll}
\multicolumn{2}{c||}{reduced}     & $t$   & $t_0$  & $t_n$    & $u_0$ & $u_n$ &
\multicolumn{2}{c||}{non-reduced} & $t$   & $t_0$  & $t_n$    & $u_0$ & $u_n$ \\ \hline
            $B_n$& \S \ref{ss:B}  & $t_l$ & $1$    & $t_s$    & $1$   & $t_s$ &
     $(BC_n,C_n)$& \S \ref{ss:BCC}& $t_m$ & $t_l^2$& $t_s t_l$& $1$   & $t_s/t_l$ \\
       $B^\vee_n$& \S \ref{ss:Bv} & $t_s$ & $1$    & $t_l^2$  & $1$   & $1$   &
$(C^\vee_n,BC_n)$& \S \ref{ss:CBC}& $t_m$ & $ t_s$ & $t_s t_l$& $t_s$ & $t_s/t_l$ \\
            $C_n$& \S \ref{ss:C}  & $t_s$ & $t_l^2$& $t_l^2$  & $1$   & $1$   &
 $(B^\vee_n,B_n)$& \S \ref{ss:BB} & $t_m$ & $1$    & $t_s t_l$& $1$   & $t_s/t_l$ \\
       $C^\vee_n$& \S \ref{ss:Cv} & $t_l$ & $t_s$  & $t_s$    & $t_s$ & $t_s$ \\
           $BC_n$& \S \ref{ss:BC} & $t_m$ & $t_l^2$& $t_s$    & $1$   & $t_s$ \\
            $D_n$& \S \ref{ss:D}  & $t$   & $1$    & $1$      & $1$   & $1$
\end{tabular}
\caption{Specialization table}
\label{tab:sp}
\end{table}
\end{mth}

Hereafter we refer Table \ref{tab:sp} as the \emph{specialization table}.

Let us explain how to read Theorem \ref{thm:1} and the specialization Table \ref{tab:sp}
in the case of type $C_n$.
The associated Macdonald polynomial has the parameters $q$ and two kinds of $t$'s.
The latter correspond to the two orbits of the extended affine Weyl group 
acting on the affine root system of type $C_n$, and we denote them by $t_s$ and $t_l$.
Using them, we denote the symmetric Macdonald polynomial of type $C_n$ 
by $P_\mu^C(x;q,t_s,t_l)$ with dominant weight $\mu$. 
See \S \ref{ss:C} for the detail of these symbols for type $C_n$.
We also have the Koornwinder polynomial $P_\mu(x;q,t,t_0,t_n,u_0,u_n)$ 
with the same dominant weight, whose detail will be explained in \S \ref{ss:HCC}.
Then, specializing the Noumi parameters as indicated in the type $C_n$ row 
in Table \ref{tab:sp}, we obtain $P_\mu^C(x;q,t_s,t_l)$. 
In other words, the following identity holds.
\begin{align}\label{eq:intro:PC}
 P_\mu^C(x;q,t_s,t_l) = P_\mu(x;q,t_s,t_l^2,t_l^2,1,1).
\end{align}
See Proposition \ref{prp:C} for the detail of type $C_n$.

We derive each of the specializations in \S \ref{ss:C} and \S \ref{ss:oth}, 
as indicated in the specialization Table \ref{tab:sp}.
Our argument is based on the fact that each family of Macdonald-Koornwinder polynomials is 
uniquely determined by the inner product.
Thus, the desired specialization will be obtained by studying the degeneration of 
the weight function of the inner product,
which is actually described in the formula \cite[(5.1.7)]{M}.
See \eqref{eq:Delta:sp} for the precise statement.
As commented at \cite[(5.1.7)]{M}, all we have to do is to 
take care the correspondence of the orbits of the extended affine Weyl group.

In \S \ref{s:RY}, as a verification of the specializing Table \ref{tab:sp}, 
we check the obtained specializations by using 
explicit formulas of Macdonald-Koornwinder polynomials.
We focus on \emph{Ram-Yip type formulas} \cite{RY,OS} which were mentioned before.
These formulas give explicit description of the coefficients in the monomial expansion
of non-symmetric Macdonald-Koornwinder polynomials as a summation of terms over 
the so-called \emph{alcove walks}, the notion introduced by Ram \cite{R}.
We do this check for Ram-Yip formulas of type $B,C$ and $D$ in the sense of \cite{RY}.
The check is done just in case-by-case calculation, 
but since the situation is rather complicated due to the notational problem of
affine root systems and parameters, we believe that it has some importance.
The result is as follows.

\begin{mth}[{Propositions \ref{prp:RY:B}, \ref{prp:RY:C} and \ref{prp:RY:D}}]
For each $\mu \in P_{C_n}$, we have 
\begin{align}
 E_\mu(x;q,t_m^{\RY},1,t_l^{\RY},1,t_l^{\RY})&=E_\mu^{B,\RY}(x;q,t_m^{\RY},t_l^{\RY}), \\
 E_\mu(x;q,t_m^{\RY},1,t_s^{\RY},1,1) &= E_\mu^{C,\RY}(x;q,t_s^{\RY},t_m^{\RY}), \\
 E_\mu(x;q,t,1,1,1,1) &= E_\mu^{D,\RY}(x;q,t).
\end{align}
\end{mth}

Here the left hand sides denote specializations of the non-symmetric Koornwinder 
polynomials $E_\mu(x)$, and the right hand side denotes the non-symmetric Macdonald 
polynomials of type $B,C$ and $D$ in the sense of \cite{RY}.
For the detail, see the beginning of \S \ref{s:RY} for the explanation.
Comparing these identities with the specialization Table \ref{tab:sp},
we find that $E_\mu^{B,\RY}(x)$ is equivalent to the polynomial of type $B_n$,
$E_\mu^{C,\RY}(x)$ is to that of type $C_n^\vee$, and $E_\mu^{D,\RY}(x)$ is to 
that of type $D_n$ in the sense of Definition \ref{dfn:type},

\subsubsection*{Notation and terminology}

Here are the notations and terminology used throughout in this paper. 
\begin{itemize}[nosep]
\item
We denote by $\bZ$ the ring of integers,
by $\bN = \bZ_{\ge 0} := \{0,1,2,\ldots\}$ the set of non-negative integers,
by $\bQ$ the field of rational numbers, and  by $\bR$ the field of real numbers.

\item
We denote by $e$ the unit of a group.

\end{itemize}

\section{Specialization table of Koornwinder polynomials}\label{s:table}

The aim of this section is to give the detail of the specialization Table \ref{tab:sp}.
As explained in \S \ref{s:intro}, we use the affine root systems 
in the sense of Macdonald \cite{M1,M}.
Our main system is that of type $(C_n^\vee,C_n)$, which will be denoted by $S$.
See \eqref{eq:S} for the precise definition.
According to the list of affine root systems in \cite[\S 1.3]{M},
those in Table \ref{tab:sp} are subsystems of $S$.
Explicitly, the following types are the subsystems of type $(C_n^\vee,C_n)$.
\begin{align}\label{eq:subsys}
  B_n, \ B_n^\vee, \ C_n, \ C_n^\vee, \ D_n \ BC_n, \ 
 (BC_n,C_n), \ (C_n^\vee,BC_n), \ (B_n^\vee,B_n).
\end{align}
The details of these subsystems will be explained in \S \ref{ss:C} and \S \ref{ss:oth}.

\subsection{Affine root system of type $(C_n^\vee,C_n)$}\label{ss:CvC}

Here we introduce the notation for the affine root system of type $(C_n^\vee,C_n)$
in the sense of Macdonald \cite{M}, following Chapter 1 of loc.\ cit.

Let $n \in \bZ_{\ge 2}$, and $E$ be the $n$-dimensional Euclidean space 
with inner product $\br{\cdot, \cdot}$.
We take and fix an orthonormal basis $\{\ep_i \mid i=1,2,\dotsc,n\}$ of $E$.
Thus, we may identify $E=(V, \br{\cdot,\cdot})$ with $V=\oplus_{i=1}^n \bR \ep_i$.
Let $F$ be the $\bR$-linear space of affine linear functions $E \to \bR$.
The inner product $\br{\cdot,\cdot}$ yields the isomorphism $F \sto V \oplus \bR c$, 
where $c$ is the constant function $c(v)=1$ for any $v \in V$.
Hereafter we identify $F$ and $V \oplus \bR c$ by this isomorphism.

We denote by $S$ the affine root system of type $(C_n^\vee,C_n)$ in the sense of 
\cite[\S 1.3, (1.3.18)]{M}. Thus, $S$ is a subset of $F=V \oplus \bR c$ given by 
\begin{align}\label{eq:S}
\begin{split}
&  S  = O_1 \sqcup O_2 \sqcup \dotsb \sqcup O_5, \\
&O_1 := \{\pm \ep_i + r c \mid 1 \le i \le n, \, r \in \bZ\}, \ 
 O_2 := 2 O_1, \ 
 O_3 := O_1+\thf c, \ 
 O_4 := 2 O_3 = O_2 + c, \\
&O_5 := \{\pm \ep_i \pm \ep_j + r c \mid 1 \le i < j \le n, \, r \in \bZ\}.
\end{split}
\end{align}
An element of $S$ is called an affine root, or just a root.
Following the choice of \cite{YK}, we consider the affine roots
\begin{align}\label{eq:ai}
 a_0 := -2\ep_1+c, \quad 
 a_j :=  \ep_j-\ep_{j+1} \ (1 \le j \le n-1), \quad 
 a_n :=  2\ep_n.
\end{align}
Thy form a basis of $S$ in the sense of \cite[\S 1.2]{M}. Obviously we have 
\begin{align}
 \thf a_0 \in O_3, \quad a_0 \in O_4, \quad a_j \in O_5 \quad (1 \le j \le n-1), \quad
 \thf a_n \in O_1, \quad a_n \in O_2.
\end{align}
Below is the Dynkin diagram cited from \cite[(1.3.18)]{M}.
The mark $*$ above the index $i$ implies that $a_i, \hf a_i \in S$.
\begin{align}
 \dynkin[mark=o, edge length=.75cm, reverse arrows,
         labels={0,1,2,,n-1,n}, labels*={*,,,,,*}] C[1]{}
\end{align}
In fact, the description \eqref{eq:S} gives the orbit decomposition of $S$ 
by the action of the extended affine Weyl group. 
For the explanation, we need to introduce more symbols.

The inner product $\br{\cdot,\cdot}$ on $V$ is extended to $F=V \oplus \bR c$ by
\begin{align}
 \br{v+r c,w+s c} := \br{v,w}, \quad v,w \in V, \quad r,s \in \bR.
\end{align}
For a non-constant function $f \in F \setminus \bR c$, 
we define $s_f \in \GL_{\bR}(F)$ by 
\begin{align}
 F \ni g \lmto s_f(g) := g-\br{g,f^\vee}f, \quad f^\vee := \tfrac{2}{\br{f,f}}f.
\end{align}
It is the reflection with respect to the hyperplane $H_f:=f^{-1}(\pr{0}) \subset V$.
Now we consider the subset 
\begin{align}\label{eq:R}
 R := \pr{\pm  \ep_i \pm \ep_j \mid 1 \le i < j \le n} \cup 
      \pr{\pm 2\ep_i \mid 1 \le i \le n} \subset S \cap V,
\end{align}
which is in fact the finite root system of type $C_n$.
Among the affine roots $a_i$ in \eqref{eq:ai}, those except $a_0$ belong to $R$,
which are the simple roots of type $C_n$.
Then the finite Weyl group $W_0$ is the subgroup 
\begin{align}\label{eq:W0}
 W_0 := \langle s_i \  (i=1,2,\dotsc,n) \rangle \subset \GL_{\bR}(V), \quad 
 s_i := s_{a_i}.
\end{align}
Note that each element in $W_0$ is an isometry for the inner product $\br{\cdot,\cdot}$.

Next, for $v \in V$, we define $t(v) \in \GL_\bR(F)$ by 
\begin{align}\label{eq:t(v)}
 F \ni f \lmto t(v)(f) := f-\br{f,v}c.
\end{align}
Then, for $w \in W_0$, we have
\begin{align}\label{eq:w.t}
 w \, t(v) \, w^{-1} = t(w v).
\end{align}
Let $P_{C_n} \subset F$ be given by 
\begin{align}\label{eq:PCn}
 P_{C_n} := \bZ\ep_1 \oplus \bZ\ep_2 \oplus \dotsb \oplus \bZ\ep_n,
\end{align}
which is in fact the weight lattice of the finite root system of type $C_n$. Then, 
\begin{align}\label{eq:tPCn}
 t(P_{C_n}) := \pr{t(\mu) \mid \mu\in P_{C_n}} \subset \GL_\bR(F)
\end{align}
is isomorphic to the additive group $P_{C_n}$.
Viewing \eqref{eq:w.t} as an action of $W_0$ on $t(P_{C_n})$, 
we can take the semigroup of \eqref{eq:W0} and \eqref{eq:tPCn} to obtain
the extended affine Weyl group $W$ of type $(C_n^\vee,C_n)$:
\begin{align}\label{eq:W}
 W := t(P_{C_n}) \rtimes W_0 \subset \GL_{\bR}(F).
\end{align}
It acts on $S$ by permutation \cite[(1.4.6), (1.4.7)]{M},
and the orbits are given in \eqref{eq:S} \cite[1.5]{M}.

Let us also give a description of $W$ as an abstract group. We set 
\begin{align}\label{eq:s0}
 s_0 := t(\ep_1)s_{2\ep_1} \in W.
\end{align}
Then $W$ has a presentation with generators 
\begin{align}\label{eq:W:gr}
 W = \br{s_0,s_1,\dotsc, s_n}
\end{align}
and the following relations.
\begin{align}\label{eq:Wrel}
\begin{split}
 s_i^2 = 1                                         \quad &(0 \le i \le n),   \\
 s_i s_j = s_j s_i                                 \quad &(\abs{i-j}>1),     \\
 s_j s_{j+1} s_j = s_{j+1} s_j s_{j+1}             \quad &(1 \le j \le n-2), \\
 s_i s_{i+1} s_i s_{i+1} = s_{i+1} s_i s_{i+1} s_i \quad &(i=0,n-1).
 \end{split}
\end{align}
Hereafter the length $\ell(w)$ of $w \in W$ indicates that 
for a reduced expression in terms of the generators $\{s_i \}_{i=0}^n$.
For later use, we write down a reduced expression of $t(\ep_i)$:
\begin{align}\label{eq:tep}
 t(\ep_i) = s_{i-1} s_{i-2} \dotsm s_1 s_0 s_1 s_2 \dotsm 
            s_n s_{n-1} s_{n-2} \dotsm s_{i+1} s_i \quad (1 \le i \le n).
\end{align}

Let us also introduce $F_\bZ \subset F$ by 
\begin{align}\label{eq:FbZ}
 F_\bZ := P_{C_n} \oplus \thf\bZ c.
\end{align}
Then we have $S \subset F_\bZ$.
We write down the action of $W$ on $F_\bZ$:
\begin{align}
&s_0(\ep_i)=\begin{cases} c-\ep_1 & (i=1)  \\ \ep_i  & (i \neq 1)\end{cases}, & 
&s_j(\ep_i)=\begin{cases}   \ep_j & (i=j+1)\\ \ep_j+1& (i   =  j) \\ 
                            \ep_i & (i \neq j, j+1) \end{cases} \quad (1 \le j \le n-1),\\
&s_n(\ep_i)=\begin{cases}  -\ep_n & (i=n)  \\ \ep_i  & (i \neq n)\end{cases}, & 
&s_k(c)=c \quad (0 \le k \le n).
\end{align}
By these formulas, we can check the orbit decomposition \eqref{eq:S} directly.

Closing this part, we recall the positive and negative parts of $S$. Let us write $S$ as 
\begin{align}
 S = \{\pm \ep_i+\tfrac{1}{2}r c, \, \pm 2\ep_i+r c \mid 1 \le i \le n, \, r \in \bZ \}
\cup \{\pm \ep_i \pm \ep_j + r c                \mid 1 \le i < j \le n, \, r \in \bZ \}.
\end{align}
It has the decomposition $S=S^+ \sqcup S^-$ with the sets $S^\pm$ of 
positive and negative roots, respectively.
To describe $S^\pm$, let us recall the decomposition of the finite root system $R$ 
of type $C_n$ (see \eqref{eq:R}) into positive and negative roots:
\begin{align}
 R = R_+\sqcup R_-, \quad 
 R_+ := \pr{2\ep_i \mid 1 \le i \le n} \cup 
        \pr{ \ep_i \pm \ep_j \mid 1 \le i < j \le n } \subset R, \quad 
 R_-:=-R_+.
\end{align}
Then, the sets $S_\pm$ are given by 
\begin{align}
 S^+:=\{\alpha+r c, \, \alpha^\vee+\thf r c \mid \alpha \in R_+, \ r \in \bN\} \cup
      \{\alpha+r c, \, \alpha^\vee+\thf r c \mid \alpha \in R_-, \ r \in \bN\},\quad 
 S^-:=-S^+.
\end{align}
Using \eqref{eq:ai}, we have $a_i \in S^+$ for each $i=0,1,\dotsc,n$. Moreover we have
\begin{align}\label{eq:S^+}
 S^+ = \sum_{i=0}^n \bN a_i \setminus \{0\}.
\end{align}
We also define $\wt{R}, \wt{R}_{\pm} \subset S$ by 
\begin{align}\label{eq:wtR}
 \wt{R} := \wt{R}_+\sqcup \wt{R}_-, \quad 
 \wt{R}_+ := \pr{\ep_i,   2\ep_i \mid 1 \le i \le n} \cup 
             \pr{\ep_i \pm \ep_j \mid 1 \le i < j \le n },\quad
 \wt{R}_- := -\wt{R}_+.
\end{align}
Then, any $a \in S$ can be presented as $a=\alpha+r c$ 
with $\alpha \in \wt{R}$ and $r \in \thf \bZ$, and we denote 
\begin{align}\label{eq:ola}
 \ol{a} := \alpha \in \wt{R}.
\end{align}

\subsection{Parameters, weight function and Koornwinder polynomials}
\label{ss:HCC}

In this subsection, we explain the parameters and the weight function for type 
$(C_n^\vee,C_n)$, and introduce the symmetric and non-symmetric Koornwinder polynomials.
As for the parameters of Koornwinder polynomials, 
we mainly use \emph{the Noumi parameters} in \cite{N}, as mentioned in \S \ref{s:intro}.
Due to the necessity in the specialization argument,
we also give a summary of the comparison of the Noumi parameters with those 
given by Macdonald in \cite{M}, which we will refer as \emph{the Macdonald parameters}.

We begin with explanation on the parameters in \cite{M}.
Let us write again the $W$-orbits \eqref{eq:S} in  $S=O_1 \sqcup \dotsb \sqcup O_5$ 
and the affine roots $a_i$ in \eqref{eq:ai}:
\begin{align}
&O_1 := \{\pm \ep_i + r c \mid 1 \le i \le n, \, r \in \bZ\}, \quad 
 O_2 := 2 O_1, \quad 
 O_3 := O_1+\thf c, \quad 
 O_4 := 2 O_3 = O_2 + c, \\
&O_5 := \{\pm \ep_i \pm \ep_j + r c \mid 1 \le i < j \le n, \, r \in \bZ\}. \\
&a_0 := -2\ep_1+c, \quad 
 a_j :=  \ep_j-\ep_{j+1} \ (1 \le j \le n-1), \quad 
 a_n :=  2\ep_n, \\
&\thf a_0 \in O_3, \quad a_0 \in O_4, \quad a_j \in O_5 \quad (1 \le j \le n-1), \quad
 \thf a_n \in O_1, \quad a_n \in O_2.
\end{align}
We attach a parameter $k_r \in \bR$ to each $W$-orbit as
\begin{align}\label{eq:kO}
 k_r \lrto O_r \quad (r=1,2,\dotsc,5),
\end{align}
and define the label $k$ \cite[\S 1.5]{M} as a map on given by 
\begin{align}\label{eq:k}
 k: S \lto \bR, \quad k(a) := k_r \quad \text{for } a \in O_r.
\end{align}
Let $q \in \bR$ be chosen, and define the set of parameters as
\begin{align}\label{eq:q^k}
 \{q^{k(a)} \mid a \in S\} = \{q^{k_1},q^{k_2},\dotsc,q^{k_5}\}.
\end{align}
We call $q^{k_r}$'s \emph{the Macdonald parameters}.
These are used in the formulation of Koornwinder polynomials in \cite{M}.

As mentioned at \eqref{eq:intro:N} and in the beginning of this \S \ref{ss:HCC}, 
in the following argument, we will mainly use \emph{the Noumi parameters} 
\begin{align}
 t,t_0,t_n,u_0,u_n
\end{align}
introduced in \cite{N}.
As will be shown in \S \ref{sss:MN} below, we have the following relation between
the Macdonald parameters and the Noumi parameters.
\begin{align}\label{eq:MN}
 (q^{2 k_1},q^{2 k_2},q^{2 k_3},q^{2 k_4},q^{k_5}) = 
 (t_n u_n,\tfrac{t_n}{u_n},t_0 u_0,\tfrac{t_0}{u_0},t).
\end{align}
Restating by \eqref{eq:kO}, the Noumi parameters and the $W$-orbits correspond in the way 
\begin{align}\label{eq:NO}
 t_n u_n \lrto O_1, \quad t_n/u_n \lrto O_2, \quad 
 t_0 u_0 \lrto O_3, \quad t_0/u_0 \lrto O_4, \quad t \lrto O_5.
\end{align}

Now we introduce the base field for (non-symmetric) Koornwinder polynomials.
Adding the square $t^{\shf},t_i^{\shf},u_i^{\shf}$ of the Noumi parameters and 
the new parameter $q^{\shf}$, 
we define the base field $\bK$ to be the rational function field 
\begin{align}\label{eq:bK}
 \bK := \bQ(q^{\hf},t^{\hf},t_{0}^{\hf},t_{n}^{\hf},u_{0}^{\hf},u_{n}^{\hf}).
\end{align}

Next, following \cite[\S 5.1]{M}, we explain the weight function for 
(non-symmetric) Koornwinder polynomials.
Using the exponent $e$ in the sense of \cite[(1.4.5)]{M}, which is given by
$e=1$ in our $(C_n^\vee,C_n)$ case, we set 
$c_0 := e^{-1} \cdot \sum_{i=0}^n a_i = \thf c$.
Here we used the affine roots $a_i$ in \eqref{eq:ai}.
Also, using $L:=P_{C^n}$, we set 
\begin{align}
 \Lambda := L \oplus \bZ c_0 = P_{C^n} \oplus \thf \bZ 
          =   \oplus_{i=1}^n \bZ \ep_i \oplus \thf \bZ.
\end{align}
Note that we have $S \subset \Lambda$.
For each $f=\mu+r c_0 \in \Lambda$, we define 
\begin{align}\label{eq:e^f}
 e^f := e^\mu q^{r/e}=e^\mu q^r.
\end{align}
Then, for the label $k$ in \eqref{eq:k},
we define the weight function $\Delta_{S,k}$ \cite[(5.1.7)]{M} as 
\begin{align}\label{eq:Delta}
 \Delta_{S,k} := \prod_{a \in S^+} \Delta_a 
               = \prod_{a \in S^+} \frac{1-q^{k(2a)} e^a}{1-q^{k(a)}e^a}.
\end{align}
Here we used $S^+$ in \eqref{eq:S^+} and set $k(2a):=0$ if $2a \notin S$.
As explained in \cite[(5.1.14)]{M}, we can rewrite $\Delta_{S,k}$ as 
\begin{align}
   \Delta_{S,k} 
&=\prod_{r=1}^4 \prod_{a \in S^+ \cap O_r} \Delta_\alpha \cdot
                \prod_{a \in S^+ \cap O_5} \Delta_\alpha 
 =\prod_{\alpha \in R_s^+} 
  \frac{(e^{2\alpha},q e^{-2\alpha};q)_\infty}
       {\prod_{r=1}^4 (v_r e^\alpha, v_r' e^{-\alpha};q)_\infty} \cdot
  \prod_{\alpha \in R_l^+} 
  \frac{(e^\alpha,q e^{-\alpha};q)}{(q^{k_5}e^\alpha,q^{k_5+1}e^{-\alpha};q)_\infty}.
\end{align}
Here $R_s^+$ and $R_l^+$ are the set of positive and short roots 
in the finite root system $R$ of type $C_n$, respectively.
Explicitly, we have 
\begin{align}
 R_s^+ := \{\ep_i \mid 1 \le i \le n\}, \quad  
 R_l^+ := \{\ep_i \pm \ep_j \mid 1 \le i < j \le n \}.
\end{align}
We also used the following $4 \times 2$ parameters $v_1,\dotsc,v_4$ and $v_1',\dotsc,v_4'$.
\begin{align}
 (v_1, \dotsc,v_4 ) := (q^{k_1  },-q^{k_2  },q^{k_3+\hf},-q^{k_4+\hf}). \quad
 (v_1',\dotsc,v_4') := (q^{k_1+1},-q^{k_2+1},q^{k_3+\hf},-q^{k_4+\hf}).
\end{align}
Finally, as mentioned in the last part of \cite[(5.1.7)]{M},
the following relation holds for each subsystem $S^0$ of the affine root system $S$.
\begin{align}\label{eq:Delta:sp}
 \rst{\Delta_{S,k}}{k(a)-k(2a)=0 \ (a \notin S^0)} = \Delta_{S^0,k}.
\end{align}
For the complete set of the subsystems $S^0$ in $S$,
see the comment in the beginning of this \S \ref{s:table}.



The weight function $\Delta_{S,k}$ defines an inner product on the space 
\begin{align}
 \bK[x^{\pm1}] = \bK[x_1^{\pm1},x_2^{\pm1},\dotsc, x_n^{\pm 1}], \quad x_i := e^{\ep_i}
\end{align} 
of the $n$-variable Laurent polynomials, where in the last part we used \eqref{eq:e^f}.
Then, by \cite[\S 5.2]{M}, we have the family of non-symmetric Koornwinder polynomials 
\begin{align}\label{eq:Emu}
 E_\mu(x) = E_\mu(x; q,t,t_0,t_n,u_0,u_n) \in \bK[x^{\pm1}], \quad \mu \in P_{C_n},
\end{align}
as a unique orthogonal basis of the inner product on $\bK[X^{\pm1}]$
satisfying the triangular property.
Moreover, by \cite[\S 5.3]{M}, for a dominant weight $\mu$ in $P_{C_n}$, 
we have the symmetric Koornwinder polynomial 
\begin{align}\label{eq:Pmu}
 P_\mu(x) = P_\mu(x; q,t,t_0,t_n,u_0,u_n) \in \bK[x^{\pm1}]^{W_0}.
\end{align}

\subsubsection{Derivation of \eqref{eq:MN}}\label{sss:MN}

Let us derive the relation \eqref{eq:MN} between the Macdonald and Noumi parameters.
We use the notation of the affine Hecke algebra given in \cite[Chapter 4]{M}.
We make one modification: The base field $K$ is enlarged so that 
it contains $\tau_i$'s and $\tau_i'$'s defined blow, and $q^{1/2}$
(in the version of \cite{M}, it contains $q$ but doesn't $q^{1/2}$).

Let $q$ be a real number such that $0<q<1$,
and $K$ be a subfield of $\bR$ containing $q^{1/2}$.
We denote by $H$ the affine Hecke algebra associated to 
the extended affine Weyl group $W$ of \eqref{eq:W} in the sense of \cite[4.1]{M}.
It is an associative $K$-algebra generated by 
\begin{align}\label{eq:H:T_i}
 H = \br{T_0,T_1,\dots,T_n}
\end{align}
with certain defining relations, for which we refer \cite[(2.2.3)--(2.2.5)]{YK}.

\begin{rmk}\label{rmk:YTexp}
We give another description of 
the affine Hecke algebra $H$. As a $K$-linear space, it has the form
\begin{align}\label{eq:HHK}
 H = H_0 \otimes_K K[Y^{t(\ep_j)} \mid j=1,2,\dotsc,n] \simeq H_0 \otimes_K K P_{C_n},
\end{align}
where $H_0$ denotes the Hecke algebra associated to the finite Weyl group $W_0$
of type $C_n$ (see \eqref{eq:W0}), 
and $K P_{C_n}$ denotes the group algebra of the additive group $P_{C_n}$.
The commuting elements $Y^{t(\ep_j)}$'s are defined in \cite[\S 3.2]{M},
and using the reduced expressions \eqref{eq:tep}, we have the following relations 
between $Y^{t(\ep_j)}$'s and the generators $T_i$'s in \eqref{eq:H:T_i}.
\begin{align}\label{eq:Yred}
 Y^{t(\ep_j)} = T_{j-1}^{-1} \dotsm T_1^{-1} T_0 \dotsm T_{n-1} T_n T_{n-1} \dotsm T_j.
\end{align}
Note that the ordering of $T_i$'s is opposite of those in some literature, 
for example \cite[\S 2.2, p.399]{Sa}, \cite[\S 3.1, p.312]{I} and \cite{Chi}.
This discrepancy is reflected on the triangular property of 
the (non-symmetric) Koornwinder polynomials.
Namely, the choices of the ordering on the space $\bK[x^{\pm 1}]$ 
and $\bK[x^{\pm 1}]^{W_0}$, where the (non-symmetric) Koornwinder polynomials live,
are in opposite between ours and those other literature.
\end{rmk}

Recall the Macdonald parameters $q^{k_1},q^{k_2},\dotsc,q^{k_5}$ in \eqref{eq:k}.
Following \cite[(4.4.3)]{M}, we introduce the additional parameters 
$\kappa_i, \kappa_i' \in K$ for $i=0,1,\dotsc,n$ as 
\begin{align}
&k_1 = k( a_n) = \thf(\kappa_n+\kappa_n'), & 
&k_2 = k(2a_n) = \thf(\kappa_n-\kappa_n'), \\
&k_3 = k( a_0) = \thf(\kappa_0+\kappa_0'), & 
&k_4 = k(2a_0) = \thf(\kappa_0-\kappa_0'), & 
&k_5 = k( a_j) = \kappa_j =\kappa_j' \quad (1 \le j \le n-1).
\end{align}
We also introduce $\tau_i,\tau_i' \in K$ for $i=0,1,\dotsc,n$ by 
\begin{align}
 \tau_i := q^{\kappa_i/2}, \quad  \tau_i' := q^{\kappa_i'/2}.
\end{align}
By definition, we have
\begin{align}\label{eq:ktau}
 q^{k_1} = \tau_n \tau_n', \quad 
 q^{k_2} = \tau_n/\tau_n', \quad 
 q^{k_3} = \tau_0 \tau_0', \quad 
 q^{k_4} = \tau_0/\tau_0', \quad 
 q^{k_5} = \tau_j \tau_j'  \quad (1 \le j \le n-1).
\end{align}

Using the parameters $\tau_i$ and $\tau_i'$, 
we explain the basic representation $\beta$ of $H$ \cite[(4.3.10)]{M},
which actually goes back to Lusztig \cite{L}.
It is a faithful representation in the group algebra $A = K L$ of 
$L:=Q^\vee_{C_n}=Q_{B_n}=P_{C_n}$ given by 
\begin{align}\label{eq:beta}
 \beta\colon H \linj \End_{K}(K P_{C_n}), \quad 
 \beta(T_i) := \tau_i s_i + \bfb_i(1-s_i) \quad (0 \le i \le n),
\end{align}
where, expressing the element of $K P_{C_n}$ corresponding to $\alpha \in P_{C_n}$
as $e^{\alpha}$, the function $\bfb_i$ is defined by 
\begin{align}\label{eq:Lusztig}
&\quad \bfb_i = \bfb(\tau_i,\tau'_i;x^{\alpha_i}) := 
 \frac{\tau_i-\tau_i^{-1}+(\tau_i'-{\tau_i'}^{-1})x^{\alpha_i/2}}{1-x^{\alpha_i}}, \\
\nonumber
&x^{\alpha_i}:= \begin{cases} x_i/x_{i+1} = e^{\ep_i-\ep_{i+1}} & (1 \le i \le n-1) \\ 
                              q x_1^{-2} = q e^{-2\ep_1}        & (i=0) \\ 
                                x_n^2 = e^{2\ep_n}              & (i=n) \end{cases}.
\end{align}
Here the symbol $\bfb(t,u;z)$ is borrowed from \cite[(4.2.1)]{M}, 
and the symbol $x^{\alpha_i}$ is from \cite{YK}.
Note that the representation $\beta$ is well defined although the function $\bfb_i$
does not belong to the group algebra $K P_{C_n}$.

The relation of the Macdonald and Noumi parameters is obtained by the comparison between
the realizations of the basic representation in \cite{M} and \cite{N}.
Slightly extending the Noumi parameters as 
\begin{align}
 t_0, \ t_n, \ u_0, \ u_n, \ 
 t_j := t, \ u_j:=1 \ (j=1,2,\dotsc,n-1),
\end{align}
we define the function $d_i(z)$ for $i=0,1,\dotsc,n$ as 
\begin{align}
 d_i(z) := \frac{t_i^{1/2}-t_i^{-1/2}+(u_i^{1/2}-u_i^{-1/2})z^{1/2}}{1-z}.
\end{align}
Then, comparing \cite[p.52]{N} and \cite[\S 4.3]{M} 
(see also \cite[(2.2.8)--(2.2.11)]{YK}), we have the relation
\begin{align}\label{eq:b=d}
 \bfb_i  = d_i(x^{\alpha_i}).
\end{align}
This relation \eqref{eq:b=d} yields the correspondence 
\begin{align}\label{eq:tauN}
 (\tau_0,\tau_n,\tau_0',\tau_n',\tau_j=\tau_j') = 
 (t_0^\hf,t_n^\hf,u_0^\hf,u_n^\hf,t^\hf).
\end{align}
Combining it with \eqref{eq:ktau}, we obtain the relation \eqref{eq:MN}.

\subsection{Specialization to affine root system of type $C_n$}\label{ss:C}

As an illustration of deriving the specialization Table \ref{tab:sp},
we explain how to find the parameter specialization for type $C_n$:
\begin{center}
\begin{tabular}{l||lllll}
        & $t$   & $t_0$  & $t_n$    & $u_0$ & $u_n$  \\ \hline
 $C_n$  & $t_s$ & $t_l^2$& $t_l^2$  & $1$   & $1$   
\end{tabular}
\end{center}
As mentioned in \S \ref{s:intro}, we need to observe the correspondence 
of the orbits of extended affine Weyl groups of type $(C_n^\vee,C_n)$ 
and of type $C_n$.
So we start with the explanation on the description of the type $C_n$ 
as the affine root subsystem of the type $(C_n^\vee,C_n)$.

Using the description \eqref{eq:S} of the affine root system $S$ of type $(C_n^\vee,C_n)$,
let us consider the following subset $S^C$ of $S$.
\begin{align}\label{eq:S^C}
\begin{split}
 S^C := O^C_s \sqcup O^C_l, \quad
&O^C_l := O_2 \sqcup O_4 = \{\pm 2\ep_i+r \mid 1 \le i \le n, \, r \in \bZ\}, \\
&O^C_s := O_5 = \{\pm \ep_i \pm \ep_j + r \mid 1 \le i < j \le n, \, r \in \bZ\}.
\end{split}
\end{align}
It is the affine root system of type $C_n$ in the sense of \cite[\S 1.3, (1.3.4)]{M}.
The following gives a basis $\{a^C_0,a^C_1,\dotsc,a^C_n\}$ of $S^C$ 
in the sense of \cite[\S 1.2]{M}.
\begin{align}
 a^C_0 := 2a_0 = -2\ep_1+1, \quad 
 a^C_j :=  a_j =   \ep_j-\ep_{j+1} \quad (1 \le j \le n-1), \quad 
 a^C_n := 2a_n =  2\ep_n,
\end{align}
Here is the Dynkin diagram cited from \cite[(1.3.4)]{M}:
\begin{align}
 \dynkin[mark=o, edge length=.75cm, labels={}, labels*={0,1,2,,n-1,n}] C[1]{}
\end{align}
The description \eqref{eq:S^C} gives the decomposition of 
the extended affine Weyl group $W^C$. To describe it, recall the finite Weyl group $W_0$ 
of type $C_n$ in \eqref{eq:W0}, which can be rewritten as
$W_0=\br{s_{a_1^C},s_{a_2^C},\dotsc,s_{a_n^C}}$. We also denote by 
\begin{align}\label{eq:PB}
 L' = P_{C_n}^\vee = P_{B_n} := 
 \oplus_{i=1}^n \bZ \ep_i \oplus \bZ \thf(\ep_1+\cdots+\ep_n)
\end{align}
the weight lattice of finite root system of type $B_n$.
Then $W^C$ is given by 
\begin{align}\label{eq:W^C}
 W^C := W_0 \ltimes t(L') = W_0 \ltimes t(P_{B_n}),
\end{align}
and it acts on $S^C$ \cite[\S 1.4, (1.4.6), (1.4.7)]{M}. The corresponding 
$W^C$-orbits are given by the above $O^C_s$ and $O^C_l$ \cite[\S 1.5]{M}.
By \eqref{eq:S^C}, we have $a^C_0, a^C_n \in O_2 \sqcup O_4 = O^C_l$ and 
$a^C_j \in O_5=O^C_s$ ($1 \le j \le n-1$).

Next, we explain the parameters for $S^C$. Similarly as in \S \ref{ss:HCC}, 
we attach parameters $k^C_s,k^C_l \in \bR$ to the $W^C$-orbits as 
\begin{align}\label{eq:k^C}
 k^C_s \lrto O^C_s, \quad k^C_l \lrto O^C_l,
\end{align}
and define the label $k^C: S^C \to \bR$ in the same way as $k:S \to \bR$ in \eqref{eq:k}.
We also denote 
\begin{align}\label{eq:C:kN}
 t^C_l := q^{k^C_l}, \quad t^C_s := q^{k^C_s} \quad (1 \le j \le n-1),
\end{align}

We now argue that under the specialization
\begin{align}
 (t,t_0,t_n,u_0,u_n) \lmto \bigl(t^C_s, (t^C_l)^2, (t^C_l)^2, 1, 1\bigr),
\end{align}
the non-symmetric Koornwinder polynomials degenerate into 
the non-symmetric Macdonald polynomials of type $C_n$.
Recalling that both polynomials are determined uniquely by the inner products,  
or by the weight functions, we see that it is enough to check that 
the weight function $\Delta_{S,k}$ in \eqref{eq:Delta} of type $(C_n,C_n)$
degenerates to that of type $C_n$.
The latter weight function is given by \cite[(5.1.7)]{M}:
\begin{align}
 \Delta^C = \Delta_{S^C,k^C} := 
 \prod_{a \in (S^C)^+} \frac{1-q^{k^C(2a)} e^a}{1-q^{k^C(a)}e^a}.
\end{align}
Here $(S^C)^+ \subset S^C$ is the set of positive roots with respect to the basis 
$\{a^C_0,a^C_1,\dotsc,a^C_n\}$, i.e., $(S^C)^+ := \sum_{i=0}^n \bN a_i^C \setminus\{0\}$,
and $k^C: S^C \to \bR$ is the extension of the label $k^C$ (see \eqref{eq:k^C})
by $k^C(2a):=0$ ($a \notin S$).
Recalling \eqref{eq:Delta:sp}, we have 
\begin{align}
 \rst{\Delta_{S,k}}{k(a)-k(2a)=0 \ (a \in S \setminus S^C)} = \Delta_{S^C,k}.
\end{align}
Thus, the desired specialization is given by 
\begin{align}\label{eq:C:k-sp}
 k(a)-k(2a) \lmto 0 \quad (a \in S \setminus S^C), \quad 
 k(a)-k(2a) \lmto k^C(a) \quad (a \in S^C).
\end{align}
Since \eqref{eq:S^C} yields $S \setminus S^C = O_1 \sqcup O_3$, 
$S^C = O^C_s \sqcup O^C_l$, $O^C_s=O_5$ and $O^C_l=O_2 \sqcup O_4$, the map 
\eqref{eq:C:k-sp} can be rewritten in terms of $k_1,k_2,\dotsc,k_5$ and $k^C_s,k^C_l$ as
\begin{align}
 k_1-k_2, \, k_3-k_4 \lmto 0, \quad k_2, \, k_4 \lmto k^C_l. \quad k_5 \lmto k^C_s.
\end{align}
Using \eqref{eq:MN} and \eqref{eq:C:kN}, and assuming $u_0,u_n>0$, 
we can further rewrite it as 
\begin{align}
\notag
&(t_n u_n)/\tfrac{t_n}{u_n}, \, (t_0 u_0)/\tfrac{t_0}{u_0} \lmto 1, \quad 
  \tfrac{t_0}{u_0}, \, \tfrac{t_n}{u_n} \lmto (t^C_l)^2, \quad t \lmto t^C_s \\
&\iff (t,t_0,t_n,u_0,u_n) \lmto \left(t^C_s, (t^C_l)^2, (t^C_l)^2, 1, 1\right).
\label{eq:spC}
\end{align}

Now we suppress the superscript $C$ in $t^C_s$ and $t^C_l$, and denote by 
\begin{align}
 E_\mu^C(x;q,t_s,t_l), \quad \mu \in P_{C_n}
\end{align}
the non-symmetric Macdonald polynomial of type $C_n$ (Definition \ref{dfn:type}).
Similarly, for a dominant $\mu \in P_{C_n}$, we denote by $P_\mu^C(x;q,t_s,t_l)$
the symmetric Macdonald polynomials of type $C_n$ .
Then the conclusion of this \S \ref{ss:C} is:

\begin{prp}\label{prp:C}
For any $\mu \in P_{C_n}$, we have
\begin{align}
 E_\mu^C(x;q,t_s,t_l) = E_\mu(x;q,t_s,t_l^2,t_l^2,1,1).
\end{align}
Also, for a dominant weight $\mu$, we have
\begin{align}
 P_\mu^C(x;q,t_s,t_l) = P_\mu(x;q,t_s,t_l^2,t_l^2,1,1).
\end{align}
\end{prp}

The following table shows the comparison of the correspondence \eqref{eq:NO} 
between the Noumi parameters and the $W$-orbits with that \eqref{eq:k^C} 
between the parameters of type $C_n$ and the $W^C$-orbits.
\begin{center}
\begin{tabular}{r|l}
 Type $(C_n^\vee,C_n)$ & Type $C_n$ \\ \hline
 $t_n u_n \lrto O_1$   & \\
 $t_0 u_0 \lrto O_3$   & \\ \hline
 $t_n/u_n \lrto O_2$   & \multirow{2}{*}{$t^C_l \lrto O_l = O_2\sqcup O_4$} \\
 $t_0/u_0 \lrto O_4$   & \\ \hline
 $t \lrto O_5$         & $t^C_s \lrto O_s=O_5$
\end{tabular}
\end{center}

\begin{rmk}\label{rmk:C}
One may wonder whether it is possible to see the specialization \eqref{eq:spC} 
on the level of affine Hecke algebras. To clarify the point, let us denote by $H^C$
the affine Hecke algebra associated to the group $W^C$ \eqref{eq:W^C}
in the sense of \cite[4.1]{M}. It is an associative algebra over $K \subset \bR$
(see \S \ref{ss:HCC}), and as a $K$-linear space, it has the form 
$H^C = H_0 \otimes_K K[Y_C^{\lambda'} \mid \lambda' \in P_{B_n}] 
 \simeq H_0 \otimes_K K P_{B_n}$ by \cite[(4.2.7), (4.3.1)]{M}.
Here we used similar notation as in \eqref{eq:HHK}. In particular, 
$H_0$ is the Hecke algebra associated to the finite Weyl group $W_0$ of type $C_n$, 
and the part $K[Y_C^{\lambda'} \mid \lambda' \in P_{B_n}]$ is a commutative subalgebra.
Also, following \cite[(4.4.2)]{M}, we define $\tau_{C,i} = \tau_{C,i}' := q^{k^C_r/2}$, 
where $r:=s$ ($a_i \in O^C_s$) and $r:=l$ ($a_i \in O^C_l$).
Then, using the function \eqref{eq:Lusztig} with the parameters 
$\tau_{C,i}$ and $\tau_{C,i}'$ instead of $\tau_i$ and $\tau_i'$,
we have a faithful $H^C$-module 
\begin{align}
 \beta^C\colon H^C \linj \End_{K}(K P_{C_n}).
\end{align}
which is the basic representation of type $C_n$.
The basic representations $\beta$ \eqref{eq:beta} and $\beta^C$ 
sit in the following diagram.
\begin{center}
\begin{tikzcd}
 H   \arrow[d, hook, "\beta"]   & 
 H^C \arrow[d, hook, "\beta^C"] \\
 \End_K(K P_{C_n}) \arrow[r, equal] & \End_K(K P_{C_n})
\end{tikzcd}
\end{center}
One can see that the specialization \eqref{eq:spC} maps 
$\beta(T_j) \mapsto \beta^C(T_j) \quad (1 \le j \le n-1)$,
but the images of $\beta(T_i)$ is not equal to $\beta^C(T_i)$ for $i=0,n$.
Thus, it is unclear whether we can see the specialization 
on the level of affine Hecke algebras $H$ and $H^C$.
\end{rmk}

\subsection{Specialization to other subsystems}\label{ss:oth}

For all the subsystems of the affine root system $S$ of type $(C_n^\vee,C_n)$,
we can make similar arguments as in \S \ref{ss:C}, 
which will yield the specialization Table \ref{tab:sp}.
In this subsection, we list all the arguments except type $C_n$ which is already done.
Let us write again the specialization table:
\begin{center}
\begin{tabular}{ll||lllll|ll||lllll}
\multicolumn{2}{c||}{reduced}     & $t$   & $t_0$  & $t_n$    & $u_0$ & $u_n$ &
\multicolumn{2}{c||}{non-reduced} & $t$   & $t_0$  & $t_n$    & $u_0$ & $u_n$ \\ \hline
            $B_n$& \S \ref{ss:B}  & $t_l$ & $1$    & $t_s$    & $1$   & $t_s$ &
     $(BC_n,C_n)$& \S \ref{ss:BCC}& $t_m$ & $t_l^2$& $t_s t_l$& $1$   & $t_s/t_l$ \\
       $B^\vee_n$& \S \ref{ss:Bv} & $t_s$ & $1$    & $t_l^2$  & $1$   & $1$   &
$(C^\vee_n,BC_n)$& \S \ref{ss:CBC}& $t_m$ & $ t_s$ & $t_s t_l$& $t_s$ & $t_s/t_l$ \\
            $C_n$& \S \ref{ss:C}  & $t_s$ & $t_l^2$& $t_l^2$  & $1$   & $1$   &
 $(B^\vee_n,B_n)$& \S \ref{ss:BB} & $t_m$ & $1$    & $t_s t_l$& $1$   & $t_s/t_l$ \\
       $C^\vee_n$& \S \ref{ss:Cv} & $t_l$ & $t_s$  & $t_s$    & $t_s$ & $t_s$ \\
           $BC_n$& \S \ref{ss:BC} & $t_m$ & $t_l^2$& $t_s$    & $1$   & $t_s$ \\
            $D_n$& \S \ref{ss:D}  & $t$   & $1$    & $1$      & $1$   & $1$
\end{tabular}
\end{center}

A remark is in order on the treatment of the type $BC_n$ and the non-reduced systems.
As we have seen in \S \ref{sss:MN}, the argument on the specialization to type $C_n$ 
used the extended affine Weyl group of of type $C_n$.
In contrast, as commented at the beginning of \S 5.1 and (5.1.7) in \cite{M},
we don't have the extended affine Weyl groups (or the affine Hecke algebras) 
associated to the type $BC_n$ and the non-reduced systems,
so we cannot follow the argument in \S \ref{sss:MN}.
However, the (non-symmetric) Macdonald polynomials for non-reduced systems 
are defined as the specialization of Koornwinder polynomials in \cite{M},
and thus the situations are easier than reduced systems.

\subsubsection{Type $B_n$}\label{ss:B}

For $n \in \bZ_{\ge 3}$, the following subset $S^B \subset S$ forms 
the affine root system of type $B_n$ in the sense of \cite[\S 1.3, (1.3.2)]{M}.
\begin{align}\label{eq:S^B}
\begin{split}
 S^B := O^B_s \sqcup O^B_l, \quad
&O^B_s := O_1 = \{\pm \ep_i+r \mid 1 \le i \le n, \ r \in \bZ\}, \\
&O^B_l := O_5 = \{\pm \ep_i \pm \ep_j + r \mid 1 \le i < j \le n, \, r \in \bZ\}.
\end{split} \\
&\dynkin[mark=o, edge length=.75cm, labels={}, labels*={0,1,2,,,n-1,n}] B[1]{}
\end{align}
Using the symbol $L'=P^\vee_{B_n}=P_{C_n}=\oplus_{i=1}^n \bZ \ep_i$ in \cite[1.4]{M}, 
the extended affine Weyl group is given by 
\begin{align}
 W^B := W_0^B \ltimes t(L') = W_0^B \ltimes t(P_{C_n}) \simeq W.
\end{align}
Here $W_0^B$ denotes the Weyl group of the finite root lattice $B_n$.
The group $W^B$ acts on $S^B$ by permutation, 
and the $W^B$-orbits are given by $O^B_s$ and $O^B_l$.
We attach parameters $k^B_s$ and $k^B_l$ to the $W^B$-orbits as 
\begin{align}
 O^B_s \lrto k^B_s, \quad O^B_l \lrto k^B_l,
\end{align}
and define the label $k^B$ by 
\begin{align}
 k^B: S^B \lto \bR, \quad 
 k^B(a):=k^B_s \quad (a \in O^B_s), \quad  k^B(a):=k^B_l \quad (a \in O^B_l).
\end{align}
Mimicking the relation \eqref{eq:MN}, 
we introduce the parameters of type $B_n$ by 
\begin{align}\label{eq:B:kN}
 t^B_s := q^{k^B_s}, \quad  t^B_l := q^{k^B_l}.
\end{align}
They correspond to the $W^B$-orbits as $t^B_s \lrt O^B_s$ and $t^B_l \lrt O^B_l$.

The weight function of type $B_n$ is given by 
\begin{align}
 \Delta^B = \Delta_{S^B,k^B} := 
 \prod_{a \in (S^B)^+} \frac{1-q^{k^B(2a)} e^a}{1-q^{k^B(a)}e^a}.
\end{align}
Then, \eqref{eq:Delta:sp} yields 
\begin{align}
 \rst{\Delta_{S,k}}{k(a)-k(2a)=0 \ (a \in S \setminus S^B)} = \Delta_{S^B,k}.
\end{align}
Thus the desired specialization is given by 
\begin{align}
 k(a)-k(2a) \lmto 0      \quad (a \in S \setminus S^B), \quad 
 k(a)-k(2a) \lmto k^B(a) \quad (a \in S^B).
\end{align}
By \eqref{eq:S^B}, we have $S \setminus S^B = O_2 \sqcup O_3 \sqcup O_4$, 
$S^B = O^B_s \sqcup O^B_l$, $O^B_s=O_1$ and $O^B_l=O_5$. Then, we can rewrite 
the specialization in terms of $k_1,\dotsc,k_5$ and $k^B_s,k^B_l$ as 
\begin{align}
 k_2-0, \, k_3-k_4, \, k_4-0 \lmto 0, \quad k_1 \lmto k^B_s, \quad k_5 \lmto k^B_l.
\end{align}
Using \eqref{eq:MN} and \eqref{eq:B:kN}, and assuming $t_0,u_0>0$, we have 
\begin{align}
&\tfrac{t_n}{u_n}, \, (t_0 u_0)/\tfrac{t_0}{u_0}, \, \tfrac{t_0}{u_0} \lmto 1, \quad 
  t_n u_n \lmto (t^B_s)^2, \quad  t \lmto t^B_l \\
&\iff (t,t_0,t_n,u_0,u_n) \lmto \bigl(t^B_l, 1, t^B_s, 1, t^B_s \bigr). \label{eq:spB}
\end{align}

Now we suppress the superscript $B$ in the parameters, and denote by 
\begin{align}
 E_\mu^B(x;q,t_s,t_l), \quad 
 \mu \in P_{B_n} := \oplus_{i=1}^n \bZ \ep_i \oplus \thf(\ep_1+\ep_2+\dotsb+\ep_n)
\end{align} 
the non-symmetric Macdonald polynomial of type $B_n$  (Definition \ref{dfn:type}).
Having that $P_{C_n} \subset P_{B_n}$, we conclude:

\begin{prp}\label{prp:B}
For any $\mu \in P_{C_n}$, we have
\begin{align}
 E_\mu^B(x;q,t_s,t_l) = E_\mu(x;q,t_l,1,t_s,1,t_s).
\end{align}
\end{prp}

Also, for a dominant weight $\mu$, we have
\begin{align}
 P_\mu^B(x;q,t_s,t_l) = P_\mu(x;q,t_l,1,t_s,1,t_s)
\end{align}
for the symmetric Macdonald polynomials of type $B_n$.

\begin{rmk}\label{rmk:B}
We can make a similar observation as in Remark \ref{rmk:C}.
Let us denote by $H^B$ the affine Hecke algebra 
for the extended Weyl group $W^B$ in the sense of \cite[Chap.\ 4]{M}.
As a linear space over the base field $K$, we have 
$H^B \simeq H_0 \otimes_K K P_{C_n} \simeq H$.
Denoting by $\beta^B$ the basic representation of $H^B$, we have the following diagram.
\begin{center}
\begin{tikzcd}
 H   \arrow[d, hook, "\beta"]   & 
 H^B \arrow[d, hook, "\beta^B"] \\
 \End_K(K P_{C_n}) \arrow[r,hook] & \End_K(K P_{B_n})
\end{tikzcd}
\end{center}
As in Remark \ref{rmk:C}, we can that the specialization \eqref{eq:spB} maps 
$\beta(T_j) \mapsto \beta^B(T_j)$ for $j=1,2,\dotsc,n-1$,
but the images of $\beta(T_i)$ is not equal to $\beta^B(T_i)$ for $i=0,n$.
\end{rmk}

\subsubsection{Type $B^\vee_n$}\label{ss:Bv} 

For $n \in \bZ_{\ge 3}$, the following subset $S^{B^\vee} \subset S$ forms 
the affine root system of type $B_n^\vee$ in the sense of \cite[\S 1.3, (1.3.3)]{M}.
\begin{align}\label{eq:S^Bv}
\begin{split}
 S^{B^\vee} := O^{B^\vee}_l \sqcup O^{B^\vee}_s, \quad
&O^{B^\vee}_l := O_2 = \{\pm 2 \ep_i+2 r \mid 1 \le i \le n, \ r \in \bZ\}, \\
&O^{B^\vee}_s := O_5 = \{\pm \ep_i \pm \ep_j + r \mid 1 \le i < j \le n, \, r \in \bZ\}.
\end{split} \\
&\dynkin[mark=o, edge length=.75cm, reverse arrows, 
         labels={}, labels*={0,1,2,,,n-1,n}] B[1]{}
\end{align}
Using the symbol $L=L'=P^\vee_{B_n}=P_{C_n}=\oplus_{i=1}^n \bZ \ep_i$ in \cite[1.4]{M},
the extended affine Weyl group is given by 
\begin{align}
 W^{B^\vee} := W_0^B \ltimes t(L') = W_0^B \ltimes t(P_{C_n}) \simeq W.
\end{align}
It acts on $S^{B^\vee}$, and the $W^{B^\vee}$-orbits are $O^{B^\vee}_s$ and $O^{B^\vee}_l$.
We attach parameters to these orbits as 
\begin{align}
 k^{B^\vee}_s \lrto O^{B^\vee}_s, \quad k^{B^\vee}_l \lrto O^{B^\vee}_l,
\end{align}
and define the label $k^{B^\vee}: S^{B^\vee} \to \bR$ as before.
We also introduce another set of parameters as 
\begin{align}\label{eq:Bv:kN}
 t^{B^\vee}_l := \tau_{B^\vee,n}^2 = q^{k^{B^\vee}_n}, \quad 
 t^{B^\vee}_s := \tau_{B^\vee,j}^2 = q^{k^{B^\vee}_j}  \quad (0 \le i \le n-1).
\end{align}
They correspond to the $W^{B^\vee}$-orbits as 
$t^{B^\vee}_s \lrt O^{B^\vee}_s$ and $t^{B^\vee}_l \lrt O^{B^\vee}_l$.

The weight function of type $B_n^\vee$ is given by 
\begin{align}
 \Delta^{B^\vee} = \Delta_{S^{B^\vee},k^{B^\vee}} := 
 \prod_{a \in (S^{B^\vee})^+} \frac{1-q^{k^{B^\vee}(2a)} e^a}{1-q^{k^{B^\vee}(a)}e^a}.
\end{align}
Then, \eqref{eq:Delta:sp} yields
\begin{align}
 \rst{\Delta_{S,k}}{k(a)-k(2a)=0 \ (a \in S \setminus S^{B^\vee})} = 
 \Delta_{S^{B^\vee},k}.
\end{align}
Thus the specialization from type $(C_n^\vee,C_n)$ to type $B_n^\vee$ is given by 
\begin{align}
 k(a)-k(2a) \lmto 0      \quad (a \in S \setminus S^{B^\vee}), \quad 
 k(a)-k(2a) \lmto k^B(a) \quad (a \in S^B).
\end{align}
By \eqref{eq:S^Bv}, we have $S \setminus S^{B^\vee} = O_1 \sqcup O_3 \sqcup O_4$, 
$S^{B^\vee} = O^{B^\vee}_s \sqcup O^{B^\vee}_l$, 
$O^{B^\vee}_s=O_5$ and  $O^{B^\vee}_l=O_2$.
Then the above specialization can be written as 
\begin{align}
 k_1-k_2, \, k_3-k_4, \, k_4-0 \lmto 0, \quad 
 k_2 \lmto k^{B^\vee}_l, \quad k_5 \lmto k^{B^\vee}_s.
\end{align}
Using \eqref{eq:MN} and \eqref{eq:Bv:kN}, and assuming $t_0,u_n,u_0>0$, 
we can further rewrite it as 
\begin{align}
&(t_n u_n)/\tfrac{t_n}{u_n}, \, (t_0 u_0)/\tfrac{t_0}{u_0}, \, \tfrac{t_0}{u_0} \lmto 1,
 \quad t_n/u_n \lmto (t^{B^\vee}_l)^2, \quad  t \lmto t^{B^\vee}_s \\
&\iff (t,t_0,t_n,u_0,u_n) \lmto \bigl(t^{B^\vee}_s, 1, (t^{B^\vee}_l)^2, 1, 1\bigr).
\end{align}

Now we suppress the superscript $B^\vee$ in the parameters, and denote by 
\begin{align}
 E_\mu^{B^\vee}(x;q,t_s,t_l), \quad \mu \in P^\vee_{B_n} = P_{C_n}
\end{align} 
the non-symmetric Macdonald polynomial of type $B_n^\vee$ 
(Definition \ref{dfn:type}).
The conclusion of this \S \ref{ss:Bv} is:

\begin{prp}
For any $\mu \in P_{C_n}$, we have
\begin{align}
 E_\mu^{B^\vee}(x;q,t_s,t_l) = E_\mu(x;q,t_s,1,t_l^2,1,1).
\end{align}
\end{prp}

\subsubsection{Type $C^\vee_n$}\label{ss:Cv}

For $n \in \bZ_{\ge 2}$, the following subset $S^{C^\vee} \subset S$ forms 
the affine root system of type $C_n^\vee$ in the sense of \cite[\S 1.3, (1.3.5)]{M}.
\begin{align}\label{eq:S^Cv}
\begin{split}
 S^{C^\vee} := O^{C^\vee}_s \sqcup O^{C^\vee}_l, \quad
&O^{C^\vee}_s := O_1 \sqcup O_3 = \{\pm \ep_i+\thf r \mid 1 \le i \le n, \ r \in \bZ\}, \\
&O^{C^\vee}_l := O_5 = \{\pm \ep_i \pm \ep_j + r \mid 1 \le i < j \le n, \, r \in \bZ\}.
\end{split} \\
&\dynkin[mark=o, edge length=.75cm, reverse arrows, 
         labels={}, labels*={0,1,2,,n-1,n}] C[1]{}
\end{align}
Using 
$L=L'=P_{C_n}^\vee=P_{B_n}=\oplus_{i=1}^n \bZ \ep_i \oplus \thf(\ep_1+\cdots+\ep_n)$,
the extended affine Weyl group is given by 
\begin{align}
 W^{C^\vee} := W_0 \ltimes t(L') = W_0 \ltimes t(P_{B_n}) = W^C.
\end{align}
The $W^{C^\vee}$-orbits on $S^{C^\vee}$ are $O^{C^\vee}_s$ and $O^{C^\vee}_l$.
We define the label $k^{C^\vee}: S^{C^\vee} \to \bR$ using the correspondence 
\begin{align}
 k^{C^\vee}_s \lrto O^{C^\vee}_s, \quad k^{C^\vee}_l \lrto O^{C^\vee}_l.
\end{align}
Mimicking the relation \eqref{eq:MN}, we define another set of parameters as 
\begin{align}\label{eq:Cv:kN}
 t^{C^\vee}_s := q^{k^{C^\vee}_s}, \quad 
 t^{C^\vee}_l := q^{k^{C^\vee}_l}.
\end{align}
They correspond to the $W^{C^\vee}$-orbits as 
$t^{C^\vee}_s \lrt O^{C^\vee}_s$ and $t^{C^\vee}_l \lrt O^{C^\vee}_l$.

The weight function is given by 
\begin{align}
 \Delta^{C^\vee} = \Delta_{S^{C^\vee},k^{C^\vee}} := 
 \prod_{a \in (S^{C^\vee})^+} \frac{1-q^{k^{C^\vee}(2a)} e^a}{1-q^{k^{C^\vee}(a)}e^a}.
\end{align}
Then \eqref{eq:Delta:sp} yields 
\begin{align}
 \rst{\Delta_{S,k}}{k(a)-k(2a)=0 \ (a \in S \setminus S^{C^\vee})} 
 = \Delta_{S^{C^\vee},k}.
\end{align}
Thus the specialization to type $C_n^\vee$ is given by 
\begin{align}
 k(a)-k(2a) \lmto 0      \quad (a \in S \setminus S^{C^\vee}), \quad 
 k(a)-k(2a) \lmto k^B(a) \quad (a \in S^B).
\end{align}
By \eqref{eq:S^Cv}, we have $S \setminus S^{C^\vee} = O_2 \sqcup O_4$, 
$S^{C^\vee} = O^{C^\vee}_s \sqcup O^{C^\vee}_l$, 
$O^{C^\vee}_s=O_1 \sqcup O_3$ and $O^{C^\vee}_l=O_5$. 
Then we can rewrite the above specialization as 
\begin{align}
 k_2-0, \, k_4-0 \lmto 0, \quad 
 k_1, \, k_3 \lmto k^{C^\vee}_s, \quad k_5 \lmto k^{C^\vee}_l.
\end{align}
Using \eqref{eq:MN} and \eqref{eq:Cv:kN}, we can rewrite it as 
\begin{align}
&\tfrac{t_n}{u_n}, \, \tfrac{t_0}{u_0} \lmto 1, \quad 
 t_n u_n, \, t_0 u_0 \lmto (t^{C^\vee}_s)^2, \quad  t \lmto t^{C^\vee}_l \\
&\iff (t,t_0,t_n,u_0,u_n) \lmto 
 \bigl(t^{C^\vee}_l, t^{C^\vee}_s, t^{C^\vee}_s, t^{C^\vee}_s, t^{C^\vee}_s \bigr).
\end{align}

We suppress the superscript $C^\vee$ in the parameters, and denote by 
\begin{align}
 E_\mu^{C^\vee}(x;q,t_s,t_l), \quad \mu \in P^\vee_{C_n} = P_{B_n}
\end{align} 
the non-symmetric Macdonald polynomial of type $C_n^\vee$ (Definition \ref{dfn:type}).
Noting that $P_{C_n} \subset P_{B_n}$, we have the conclusion:

\begin{prp}
For any $\mu \in P_{C_n}$, we have
\begin{align}
 E_\mu^{C^\vee}(x;q,t_s,t_l) = E_\mu(x;q,t_l,t_s,t_s,t_s,t_s).
\end{align}
\end{prp}

\subsubsection{Type $BC_n$}\label{ss:BC}

For $n \in \bZ_{\ge 1}$, the following subset $S^{BC} \subset S$ forms 
the affine root system of type $BC_n$ in the sense of \cite[\S 1.3, (1.3.6)]{M}.
\begin{align}\label{eq:S^BC}
\begin{split}
 S^{BC} := O^{BC}_s \sqcup O^{BC}_m \sqcup O^{BC}_l, \quad
&O^{BC}_s := O_1 = \{\pm  \ep_i+ r    \mid 1 \le i \le n, \ r \in \bZ\}, \\
&O^{BC}_l := O_4 = \{\pm 2\ep_i+ 2r+1 \mid 1 \le i \le n, \ r \in \bZ\}, \\
&O^{BC}_m := O_5 = \{\pm  \ep_i \pm \ep_j + r \mid 1 \le i < j \le n, \, r \in \bZ\}.
\end{split}
\end{align}
Hereafter we assume $n \ge 2$ to make the argument compatible with that so far.
The Dynkin diagram is then given by 
\begin{align}\label{dynkin:BC}
 \dynkin[mark=o, edge length=.75cm, reverse arrows, 
         labels={}, labels*={0,1,2,,,n-1,n}] A[2]{even}
\end{align}

Recall the comment in the beginning of this \S \ref{ss:oth}.
We will not introduce a new extended affine Weyl group, 
but consider the group $W$ of type $(C_n^\vee,C_n)$ (see \eqref{eq:W}).
It acts on $S^{BC}$, and the $W$-orbits are given by 
$O^{BC}_s$, $O^{BC}_m$ and $O^{BC}_l$.
Hence, we already have the correspondence between the Macdonald parameters
of type $(C_n^\vee,C_n)$ and the $W$-orbits on $S^{B C}$.
Let us denote 
\begin{align}\label{eq:BC:kN}
 t^{BC}_s := q^{k_1}, \quad t^{BC}_m := q^{k_5}, \quad t^{BC}_l := q^{k_4},
\end{align}
which correspond to the $W$-orbits $O^{BC}_s$, $O^{BC}_m$ and $O^{BC}_l$, respectively.

Following \cite[(5.1.77)]{M}, we define the weight function $\Delta_{S^{BC},k}$ 
of type $BC_n$ to be the specialization of $\Delta_{S,k}$ of type $(C_n^\vee,C_n)$.
In other words, we take the right hand side of \eqref{eq:Delta:sp} as the definition:
\begin{align}
 \Delta^{BC} = \Delta_{S^{BC},k} := 
 \rst{\Delta_{S,k}}{k(a)-k(2a)=0 \ (a \in S \setminus S^{BC})}.
\end{align}
By \eqref{eq:S^BC}, we have $S \setminus S^{BC} = O_2 \sqcup O_3$ and 
$S^{BC} = O^{BC}_s \sqcup O^{BC}_m \sqcup O^{BC}_l = O_1 \sqcup O_5 \sqcup O_4$.
Them, we can see that the specialization to type $BC_n$ is given by 
\begin{align}
 k_2-0, \, k_3-k_4 \lmto 0.
\end{align}
Using \eqref{eq:MN} and \eqref{eq:BC:kN}, we can rewrite it as 
\begin{align}
&\tfrac{t_n}{u_n}, \, (t_0 u_0)/\tfrac{t_0}{u_0} \lmto 1, \quad 
 t_n u_n \lmto (t^{BC}_s)^2, \quad 
 \tfrac{t_0}{u_0} \lmto (t^{BC}_l)^2, \quad  t \lmto t^{BC}_m \\
&\iff (t,t_0,t_n,u_0,u_n) \lmto 
 \bigl(t^{BC}_m, (t^{BC}_l)^2, t^{BC}_s, 1, t^{BC}_s \bigr).
\end{align}

Now we suppress the superscript $BC$ in the parameters, and denote by 
\begin{align}
 E_\mu^{BC}(x;q,t_s,t_m,t_l), \quad \mu \in P_{C_n}
\end{align} 
the non-symmetric Macdonald polynomial of type $BC_n$ (Definition \ref{dfn:type}).
Then the conclusion is:

\begin{prp}
For any $\mu \in P_{C_n}$, we have
\begin{align}
 E_\mu^{BC}(x;q,t_s,t_m,t_l) = E_\mu(x;q,t_m,t_l^2,t_s,1,t_s).
\end{align}
\end{prp}

\subsubsection{Type $D_n$}\label{ss:D}

For $n \in \bZ_{\ge 4}$, the following subset $S^D \subset S$ forms 
the affine root system of type $D_n$ in the sense of \cite[\S 1.3, (1.3.7)]{M}.
\begin{gather}\label{eq:S^D}
 S^D := O_5 = \{\pm \ep_i \pm \ep_j + r \mid 1 \le i < j \le n, \, r \in \bZ\}. \\
 \dynkin[mark=o, edge length=.75cm, reverse arrows, 
         labels={}, labels*={0,1,2,,,n-2,n-1,n}] D[1]{}
\end{gather}
Using the Weyl group $W_0^D$ and the weight lattice 
\begin{align}\label{eq:PD}
 L' = P_{D_n} := 
 \bZ\ep_1 \oplus \dotsb \oplus \bZ\ep_n \oplus \bZ\thf(\ep_1+\dotsb+\ep_n)
\end{align}
of the finite root system of type $D_n$, the extended affine Weyl group is given by 
\begin{align}\label{eq:WD}
 W^D := W_0^D \ltimes t(L') = W_0^D \ltimes t(P_{D_n}).
\end{align}
It acts on $S^D$ by permutation, and there is a unique orbit.
Attaching $k^D \in \bR$ to this unique orbit, 
we define the label by $k^D(a):=k^D$ ($a \in S^D$), and introduce
\begin{align}\label{eq:D:kN}
 t^D := q^{k^D}.
\end{align}

The weight function is given by 
\begin{align}
 \Delta^D = \Delta_{S^D,k^D} := 
 \prod_{a \in (S^D)^+} \frac{1-q^{k^B(2a)} e^a}{1-q^{k^B(a)}e^a}.
\end{align}
The relation \eqref{eq:Delta:sp} yields 
\begin{align}
 \rst{\Delta_{S,k}}{k(a)-k(2a)=0 \ (a \in S \setminus S^D)} = \Delta_{S^D,k}.
\end{align}
Thus, the specialization to type $D_n$ is given by 
\begin{align}
 k(a)-k(2a) \lmto 0      \quad (a \in S \setminus S^D), \quad 
 k(a)-k(2a) \lmto k^D(a) \quad (a \in S^D).
\end{align}
By \eqref{eq:S^D}, we have $S \setminus S^D = O_1 \sqcup \cdots \sqcup O_4$ and 
$S^B = O_5$. Then, we can rewrite the specialization in terms of 
$k_1,\dotsc,k_5$ and $k^D$ as 
\begin{align}
 k_2-0, \, k_3-k_4, \, k_4-0 \lmto 0, \quad k_1 \lmto k^B_s, \quad k_5 \lmto k^B_l.
\end{align}
Using \eqref{eq:MN} and \eqref{eq:D:kN}, we have 
\begin{align}
&t_n u_n, \, t_n/u_n, \, t_0 u_0, \, t_0/u_0 \lmto 1, \quad 
 t \lmto t^D 
 \iff (t,t_0,t_n,u_0,u_n) \lmto \left(t^D, 1, 1, 1, 1\right).
\end{align}

We suppress the superscript $D$ in the parameters, and denote by 
\begin{align}
 E_\mu^D(x;q,t), \quad \mu \in P_{D_n}
\end{align} 
the non-symmetric Macdonald polynomial of type $D_n$ (Definition \ref{dfn:type}).
Since $P_{C_n} \subset P_{D_n}$, we have:

\begin{prp}\label{prp:D}
For any $\mu \in P_{C_n}$, we have
\begin{align}
 E_\mu^D(x;q,t) = E_\mu(x;q,t,1,1,1,1).
\end{align}
\end{prp}

\subsubsection{Type $(BC_n,C_n)$}\label{ss:BCC}

For $n \in \bZ{\ge 1}$, the following subset $S^{BC,C} \subset S$ forms 
the affine root system of type $(BC_n,C_n)$ in the sense of \cite[\S 1.3, (1.3.15)]{M}.
\begin{align}\label{eq:S^BCC}
\begin{split}
&S^{BC,C} := O^{BC,C}_s \sqcup O^{BC,C}_m \sqcup O^{BC,C}_l, \\
&O^{BC,C}_s := O_1 = \{\pm  \ep_i+ r  \mid 1 \le i \le n, \ r \in \bZ\}, \\
&O^{BC,C}_l := O_2 \sqcup O_4 = \{\pm 2\ep_i+ r \mid 1 \le i \le n, \ r \in \bZ\}, \\
&O^{BC,C}_m := O_5 = \{\pm  \ep_i \pm \ep_j + r \mid 1 \le i < j \le n, \, r \in \bZ\}.
\end{split} \\
&\dynkin[mark=o, edge length=.75cm, reverse arrows, 
         labels*={,,,,,,*}, labels={0,1,2,,,n-1,n}] A[2]{even}
\end{align}
The diagram is for $n \ge 2$, and hereafter we assume this condition.
The mark $*$ above the index $n$ implies that there is a basis $\{a_i^{BC,C}\}_{i=0}^n$
such that $a_n^{BC,C}, 2 a_n^{BC,C} \in S^{BC,C}$.
There are three $W$-orbits $O^{BC,C}_s$, $O^{BC,C}_m$ and $O^{BC,C}_l$.
We introduce the parameters 
\begin{align}\label{eq:BC,C:kN}
 t^{BC,C}_s := q^{k_1}, \quad t^{BC,C}_m := q^{k_5}, \quad t^{BC,C}_l := q^{k_2},
\end{align}
which correspond to the $W$-orbit $O^{BC,C}_s$, $O^{BC,C}_m$ 
and $O^{BC,C}_l$, respectively.

Similarly as in the previous \S \ref{ss:BC}, the weight function $\Delta_{S^{BC,C},k}$
of type $(BC_n,C_n)$ is defined by the specialization of $\Delta_{S,k}$ as 
\begin{align}
 \Delta^{BC,C} = \Delta_{S^{BC,C},k} := 
 \rst{\Delta_{S,k}}{k(a)-k(2a)=0 \ (a \in S \setminus S^{BC,C})}.
\end{align}
By \eqref{eq:S^BCC}, we have $S \setminus S^{BC,C} = O_3$, $O^{BC,C}_l=O_2 \sqcup O_4$, 
which implies that the specialization to type $(BC_n,C_n)$ is given by 
\begin{align}
 k_3-k_4 \lmto 0, \quad k_2 \lmto k_4.
\end{align}
Using \eqref{eq:MN} and \eqref{eq:BC,C:kN}, and assuming $u_0>0$, we can rewrite it as 
\begin{align}
&(t_0 u_0)/\tfrac{t_0}{u_0}  \lmto 1, \quad 
 t_n u_n \lmto (t^{BC,C}_s)^2, \quad 
 \tfrac{t_n}{u_n}, \, \tfrac{t_0}{u_0} \lmto (t^{BC,C}_l)^2, \quad  
 t \lmto t^{BC,C}_m \\
&\iff (t,t_0,t_n,u_0,u_n) \lmto \bigl(
 t^{BC,C}_m, (t^{BC,C}_l)^2, t^{BC,C}_s t^{BC,C}_l, 1, t^{BC,C}_s/t^{BC,C}_l \bigr).
\end{align}

We suppress the superscript $BC,C$ in the parameters, and denote by 
\begin{align}
 E_\mu^{BC,C}(x;q,t_s,t_m,t_l), \quad \mu \in P_{C_n}
\end{align} 
the non-symmetric Macdonald polynomial of type $(BC_n,B_n)$ (Definition \ref{dfn:type}).
The conclusion of this \S \ref{ss:BCC} is:

\begin{prp}
For any $\mu \in P_{C_n}$, we have
\begin{align}
 E_\mu^{BC,C}(x;q,t_s,t_m,t_l) = E_\mu(x;q,t_m,t_l^2,t_s t_l,1,t_s/t_l).
\end{align}
\end{prp}

\subsubsection{Type $(C^\vee_n,BC_n)$}\label{ss:CBC}

For $n \in \bZ_{\ge 1}$, the following subset $S^{C^\vee,BC} \subset S$ forms 
the affine root system of type $(BC_n,C_n)$ in the sense of \cite[\S 1.3, (1.3.16)]{M}.
\begin{align}\label{eq:S^CvBC}
\begin{split}
&S^{C^\vee,BC} := O^{C^\vee,BC}_s \sqcup O^{C^\vee,BC}_m \sqcup O^{C^\vee,BC}_l, \\
&O^{C^\vee,BC}_s := O_1 \sqcup O_3 = 
 \{\pm  \ep_i+ \thf r  \mid 1 \le i \le n, \ r \in \bZ\}, \\
&O^{C^\vee,BC}_l := O_2 = \{\pm 2\ep_i+ 2r \mid 1 \le i \le n, \ r \in \bZ\}, \\
&O^{C^\vee,BC}_m := O_5 = \{\pm  \ep_i \pm \ep_j + r \mid 1 \le i < j \le n, \, r \in \bZ\}.
\end{split}
\end{align}
Hereafter we assume $n \ge 2$. Then the Dynkin diagram is given by 
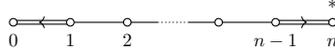
\begin{align}
 \dynkin[mark=o, edge length=.75cm, reverse arrows,
         labels={0,1,2,,n-1,n}, labels*={,,,,,*}] C[1]{}
\end{align}
There are three $W$-orbits 
$O^{C^\vee,BC}_s$, $O^{C^\vee,BC}_m$ and $O^{C^\vee,BC}_l$, and 
the parameters are defined to be 
\begin{align}\label{eq:Cv,BC:kN}
 t^{C^\vee,BC}_s := q^{k_1}, \quad t^{C^\vee,BC}_m := q^{k_5}, \quad 
 t^{C^\vee,BC}_l := q^{k_2}.
\end{align}

The weight function of type $(C^\vee_n,BC_n)$ is defined by 
\begin{align}
 \Delta^{C^\vee,BC} = \Delta_{S^{C^\vee,BC},k} := 
 \rst{\Delta_{S,k}}{k(a)-k(2a)=0 \ (a \in S \setminus S^{C^\vee,BC})}.
\end{align}
By \eqref{eq:S^CvBC}, we have $S \setminus S^{C^\vee,BC} = O_4$ and 
$O^{C^\vee,BC}=O_1 \sqcup O_3$, which implies that 
\begin{align}
 k_4-0 \lmto 0, \quad k_1 \lmto k_3
\end{align}
give the desired specialization.
Using \eqref{eq:MN} and \eqref{eq:Cv,BC:kN}, we can rewrite it as 
\begin{align}
&t_0/u_0 \lmto 1, \quad 
 t_0 u_0, \, t_n u_n \lmto (t^{C^\vee,BC}_s)^2, \quad 
 t_n/u_n \lmto (t^{C^\vee,BC}_l)^2, \quad  t \lmto t^{C^\vee,BC}_m \\
&\iff (t,t_0,t_n,u_0,u_n) \lmto 
 \bigl(t^{C^\vee,BC}_m, t^{C^\vee,BC}_s, t^{C^\vee,BC}_s t^{C^\vee,BC}_l, 
       t^{C^\vee,BC}_s, t^{C^\vee,BC}_s/t^{C^\vee,BC}_l \bigr).
\end{align}

We suppress the superscript $C^\vee,BC$ in the parameters, and denote by 
\begin{align}
 E_\mu^{C^\vee,BC}(x;q,t_s,t_m,t_l), \quad \mu \in P_{C_n}
\end{align} 
the non-symmetric Macdonald polynomial of type $(C^\vee_n,BC_n)$ 
(Definition \ref{dfn:type}). The conclusion of this \S \ref{ss:BCC} is:

\begin{prp}
For any $\mu \in P_{C_n}$, we have
\begin{align}
 E_\mu^{C^\vee,BC}(x;q,t_s,t_m.t_l) = E_\mu(x;q,t_m,t_s,t_s t_l,t_s,t_s/t_l).
\end{align}
\end{prp}

\subsubsection{Types $(C_2,C^\vee_2)$ and $(B^\vee_n,B_n)$}\label{ss:BB}

The affine root systems of type $(C_2,C^\vee_2)$ and of type $(B^\vee_n,B_n)$ 
with $n \in \bZ_{\ge 3}$ in the sense of \cite[\S 1.3, (1.3.17)]{M}
are given by the following subset $S^{B^\vee,B} \subset S$.
\begin{align}\label{eq:S^BvB}
\begin{split}
 S^{B^\vee,B} :=
&O^{B^\vee,B}_s \sqcup O^{B^\vee,B}_m \sqcup O^{B^\vee,B}_l, \\
&O^{B^\vee,B}_s := O_1 = \{\pm  \ep_i+ r  \mid 1 \le i \le n, \ r \in \bZ\}, \\
&O^{B^\vee,B}_l := O_2 = \{\pm 2\ep_i+ 2r \mid 1 \le i \le n, \ r \in \bZ\}, \\
&O^{B^\vee,B}_m := O_5 = \{\pm  \ep_i \pm \ep_j + r \mid 1 \le i < j \le n, \, r \in \bZ\}.
\end{split}
\end{align}
In the case $n \ge 3$, the Dynkin diagram is given by 
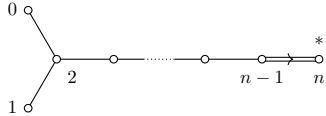
\begin{align}
&\dynkin[extended, mark=o, edge length=.75cm, 
         labels*={,,,,,,*}, labels={0,1,2,,,n-1,n}] B{}
\end{align}
The $W$-orbits are $O^{B^\vee,B}_s$, $O^{B^\vee,B}_m$ and $O^{B^\vee,B}_l$.
The corresponding parameters are defined to be
\begin{align}\label{eq:BvB:kN}
 t^{B^\vee,B}_s := q^{k_1}, \quad t^{B^\vee,B}_m := q^{k_5}, \quad 
 t^{B^\vee,B}_l := q^{k_2}.
\end{align}

The weight function $\Delta_{S^{B^\vee,B},k}$ is defined by
\begin{align}
 \Delta^{B^\vee,B} = \Delta_{S^{B^\vee,B},k} := 
 \rst{\Delta_{S,k}}{k(a)-k(2a)=0 \ (a \in S \setminus S^{B^\vee,B})}.
\end{align}
By \eqref{eq:S^BvB}, we have $S \setminus S^{B^\vee,B} = O_3 \sqcup O_4$, 
which implies that 
\begin{align}
 k_3-k_4, \, k_4-0 \lmto 0
\end{align}
gives the specialization to type $(C_2,C^\vee_2)$ and $(B^\vee_n,B_n)$.
Using \eqref{eq:MN} and \eqref{eq:BvB:kN}, we can rewrite it as 
\begin{align}
&t_0 u_0, \, t_0/u_0 \lmto 1, \quad 
 t_n u_n \lmto (t^{B^\vee,B}_s)^2, \quad 
 t_n/u_n \lmto (t^{B^\vee,B}_l)^2, \quad  t \lmto t^{B^\vee,B}_m \\
&\iff (t,t_0,t_n,u_0,u_n) \lmto \bigl(t^{B^\vee,B}_m, 1, 
 t^{B^\vee,B}_s t^{B^\vee,B}_l, 1, t^{B^\vee,B}_s/t^{B^\vee,B}_l \bigr).
\end{align}

We suppress the superscript $B^\vee,B$ in the parameters, and denote by 
\begin{align}
 E_\mu^{B^\vee,B}(x;q,t_s,t_m,t_l), \quad \mu \in P_{C_n}
\end{align} 
the non-symmetric Macdonald polynomial of types $(C_2,C_2^\vee)$ and $(B^\vee_n,B_n)$ 
(Definition \ref{dfn:type}). The conclusion of this \S \ref{ss:BCC} is:

\begin{prp}\label{prp:BvB}
For any $\mu \in P_{C_n}$, we have
\begin{align}
 E_\mu^{B^\vee,B}(x;q,t_s,t_m.t_l) = E_\mu(x;q,t_m,1,t_s t_l,1,t_s/t_l).
\end{align}
\end{prp}

\subsection{Relation to Koornwinder's specializations in admissible pairs}

As mentioned in \S \ref{s:intro}, in the original theory \cite{Mp}, 
Macdonald used admissible pairs to formulate 
his family of multivariate orthogonal polynomials for general root systems.
Here, an admissible pair means a pair $(R,S)$ of root systems satisfying
the following conditions.
\begin{itemize}[nosep]
\item 
Both $R$ and $S$ span the common finite-dimensional Euclidean space $V$.
\item
$S$ is a reduced.
\item
The Weyl groups are identical, i.e., $W_R = W_S$.
\end{itemize}

In \cite[\S 6.1]{K}, Koornwinder obtained Macdonald polynomials 
of the admissible pairs 
\begin{align}
 (R,S) = (R_{BC_n},S_{B_n}), \ (R_{BC_n},S_{C_n})
\end{align}
by specializing the parameters in his polynomials.
The parameters in \cite{K} are denoted as
\begin{align}
 a,b,c,d,t,q,
\end{align}
and we call them the Koornwinder parameters.
The finite root systems $R_{BC_n}$, $S_{B_n}$ and $S_{C_n}$ are 
\begin{align}
 R_{BC_n}&:= \{\pm \ep_i \mid 1 \le i \le n \} \sqcup \{\pm 2\ep_i \mid 1 \le i \le n \}
             \sqcup \{\pm \ep_i \pm \ep_j \mid 1 \le i < j \le n \}, \notag \\
 S_{B_n} &:= \{\pm \ep_i \mid 1 \le i \le n \} \sqcup
             \{\pm \ep_i \pm \ep_j \mid 1 \le i < j \le n \}, \notag \\
 S_{C_n} &:= \{\pm \ep_i \mid 1 \le i \le n \} \sqcup
             \{\thf(\pm \ep_i \pm \ep_j) \mid 1 \le i < j \le n \}.
 \label{eq:adp:S_C}
\end{align}
Using them, the specializations in \cite[\S 6.1]{K} are described as 
\begin{align}
\label{eq:K:BCB}
(R_{BC_n},S_{B_n}): \ 
&(a,b,c,d,t,q) \lmto (q^{1/2},-q^{1/2},a_B b_B^{1/2}, -b_B^{1/2},t_B,q), \\
\label{eq:K:BCC}
(R_{BC_n},S_{C_n}): \  
&(a,b,c,d,t,q) \lmto (a_C b_C^{1/2}, q a_C b_C^{1/2}, -b_C^{1/2}, -q b_C^{1/2}, t_C, q^2).
\end{align}
There are only given these results in \cite[\S 6.1]{K} .
We guess that they are derived by the comparison of the weight functions of inner products,
as we did in the previous \S \ref{ss:C} and \S \ref{ss:oth}.

In \cite[p.54]{N}, Noumi gave the correspondence between
the Noumi parameters $q,t,t_0,t_n,u_0,u_n$ and the Koornwinder parameters $a,b,c,d,t,q$.
The correspondence is that $q$ and $t$ are common, and 
\begin{align}
 (t_0,t_n,u_0,u_n) = (-c d/q,-a b,-c/d,-a/b). 
\end{align}
We can then rewrite the specialization \eqref{eq:K:BCB} to the admissible pair 
$(R_{BC_n},S_{B_n})$ as 
\begin{align}
 (t,t_0,t_n,u_0,u_n) \lmto (t_B,1,a_B b_B,1,a_B).
\end{align}
Thus, setting $t_B=t^{B^\vee,B}_m$, $a_B=t^{B^\vee,B}_s/t^{B^\vee,B}_l$ and
$b_B=(t^{B^\vee,B}_l)^2$, we see that it coincides with the specialization to 
type $(B_n^\vee,B_n)$ in \S \ref{ss:BB}.

Let us remark that a similar rewriting of the specialization 
\eqref{eq:K:BCC} to the admissible pair $(R_{BC_n},S_{C_n})$ 
does not have a corresponding one in Table \ref{tab:sp}.
It seems to be due to that the root system $S_{C_n}$ in \eqref{eq:adp:S_C}
cannot be treated in the formulation of \cite{M}.


\section{Specialization in Ram-Yip type formula}\label{s:RY}

In this section, we give a partial check of the specialization Table \ref{tab:sp}
in the level of \emph{Ram-Yip type formulas}.
Precisely speaking, we show that the non-symmetric Koornwinder polynomial degenerates to 
the non-symmetric Macdonald polynomials of types $B,C,D$ in the sense of \cite{RY}
by the specializations given in Table \ref{tab:sp}, 
using explicit Ram-Yip type formulas of those polynomials. 

Let us explain what we mean by the word \emph{Ram-Yip type formulas}.
In \cite{RY}, Ram and Yip derived explicit formulas of non-symmetric 
Macdonald polynomials of reduced affine root systems using alcove walks.
Their argument is designed to work in general setting, and the details are 
later given by Orr and Shimozono in \cite{OS}, which derives among many results 
an explicit formula of the non-symmetric Koornwinder polynomial. We call 
all of these formulas Ram-Yip type formulas of non-symmetric Macdonald polynomials.

A caution is now in order. The realization of affine root systems in \cite{RY} 
is different from our default one in \cite{M}.
For distinction, we denote by $S^{X,\RY}$ the affine root system of type $X$
used in \cite{RY}, and call the non-symmetric Macdonald polynomials of type $X$
treated in loc.\ cit. the polynomial \emph{of Ram-Yip type $X$}.

Let us summarize the results given in this \S \ref{s:RY} 
in the following Table \ref{tab:sp:RY}, 

\begin{table}[htbp]
\centering
{\setlength{\extrarowheight}{1pt}%
 \begin{tabular}{ll||ccccc}
             &                 &$t$         &$t_0$ &$t_n$       &$u_0$ &$u_n$ \\ \hline
 $B_n^{\RY}$ &\S \ref{ss:RY:B} &$t_m^{\RY}$ &$1$   &$t_l^{\RY}$ &$1$   &$t_l^{\RY}$ \\  
       $B_n$ &                 &$t_l$       &$1$   &$t_s$       &$1$   &$t_s$ \\ \hline
 $C_n^{\RY}$ &\S \ref{ss:RY:C} &$t_m^{\RY}$ &$1$   &$t_s^{\RY}$ &$1$   &$1$   \\
 $B_n^\vee $ &                 &$t_s$       &$1$   &$t_l^2$     &$1$   &$1$   \\ \hline
       $D_n$ &\S \ref{ss:RY:D} &$t$         &$1$   &$1$         &$1$   &$1$
 \end{tabular}}
\caption{Specialization table for Ram-Yip formulas}
\label{tab:sp:RY}
\end{table}

As mentioned above, we treat the types $B_n$, $C_n$ and $D_n$ in the sense of \cite{RY},
each in \S \ref{ss:RY:B}, \S \ref{ss:RY:C} and \S\ref{ss:RY:D}, respectively.
Since the types $B_n$ and $C_n$ have discrepancy from those in our default \cite{M},
we use the symbols $B_n^{\RY}$ and $C_n^{\RY}$ in Table \ref{tab:sp:RY}.
The type $D_n$ has no discrepancy, and we use the symbol $D_n$.
The $B_n^{\RY}$ row in Table \ref{tab:sp:RY} indicates the specialization of 
the Noumi parameters to obtain the non-symmetric polynomial of Ram-Yip type $B_n$.
More explicitly, denoting the latter by $E^{B,\RY}_\mu(x)$, we have 
\begin{align}
 E_\mu(x;q,t_m^{\RY},1,t_l^{\RY},1,t_l^{\RY}) = E^{B,\RY}_\mu(x;q,t_m^{\RY},t_l^{\RY}).
\end{align}
This equality will be shown in Proposition \ref{prp:RY:B}.
The $B_n$ row in Table \ref{tab:sp:RY} is a copy from 
the specialization Table \ref{tab:sp}, which we give in the intention of checking
the specialization argument in \S \ref{ss:C} and \S \ref{ss:oth}. 
As for the other types, Table \ref{tab:sp:RY} claims that the type $D_n$ is clean,
but that the type $C_n^{\RY}$ (Ram-Yip type $C_n$) is a little confusing, 
which turns out to correspond to the type $B_n^\vee$ in the sense of \cite{M}.

\subsection{Ram-Yip type formula of type $(C^\vee_n,C_n)$}

In this subsection, we recall Ram-Yip type formula of type $(C^\vee_n,C_n)$,
i.e., the Ram-Yip type formula of non-symmetric Koornwinder polynomial, 
derived in \cite{OS}. For the notation, we follow \cite{YK}.

\subsubsection{Alcove walks of type $(C^\vee_n,C_n)$}

Here we recall the alcove walks of type $(C^\vee_n,C_n)$ introduced in \cite{OS}.
See \cite[\S 2.1.3]{YK} for more information and illustrated examples.

We keep the notation for the affine root system $S$ of type  $(C^\vee_n,C_n)$ 
introduced in \S \ref{ss:CvC}.
Thus, the system $S$ is realized in $F=V \oplus \bR c$, $V=\oplus_{i=1}^n \bR \ep_i$
with the $W$-orbit decomposition $S=O_1 \sqcup \dotsb \sqcup O_5$.
We write again the decomposition \eqref{eq:S} and the affine roots $a_i$ in \eqref{eq:ai}:
\begin{align}
&O_1 := \{\pm \ep_i + r c \mid 1 \le i \le n, \, r \in \bZ\}, \quad 
 O_2 := 2 O_1, \quad 
 O_3 :=   O_1+\thf c, \quad 
 O_4 := 2 O_3 = O_2 + c, \\
&O_5 := \{\pm \ep_i \pm \ep_j + r c \mid 1 \le i < j \le n, \, r \in \bZ\}. \\
&a_0 := -2\ep_1+c, \quad 
 a_j :=   \ep_j-\ep_{j+1} \quad (1 \le j \le n-1), \quad 
 a_n :=  2\ep_n, 
 \label{eq:RY:a0} \\
&\thf a_0 \in O_3, \quad a_0 \in O_4, \quad a_j \in O_5 \ (1 \le j \le n-1), \quad
 \thf a_n \in O_1, \quad a_n \in O_2.
\end{align}

For each affine root $a \in S$, we denote by $H_a:=\pr{x\in V \mid a(x)=0}$
the associated hyperplanes in $V$.
An alcove is a connected component of $V \setminus \bigcup_{a \in S}H_a$.
The distinguished alcove 
\begin{equation}\label{eq:FA}
 A := \pr{x\in V \mid a_i(x)>0 \ (i=0,1,\dotsc,n)} \subset V
\end{equation}
is called the fundamental alcove.
We have a bijection 
\begin{align}
 W \ni w \longmapsto w A \in \pi_0(V \setminus \tbcup_{a \in S}H_a).
\end{align}
The boundary of an alcove $w A$ consists of $n+1$ hyperplanes,
each of which is called an edge of $w A$.
For an edge $H$ of an alcove $w A$, there exists $i \in \{0,1,\dotsc,n\}$ 
such that $H$ separates $w A$ and $w s_i A$.
The edge $H$ has two sides facing $w A$ and $w s_i A$, 
and we call them $w A$-side and $w s_i A$-side, respectively.

We assign $+$ or $-$ to each side of an edge of an alcove $w A$ as follows.
Let $\pr{H_{\gamma_i} \mid i=0,1,\dotsc,n}$ be the hyperplanes surrounding 
the alcove $wA$. Reordering $\gamma_i$'s if necessary, we can assume that
$H_{\gamma_i}$ separates $w A$ and $w s_i A$ for each $i$. 
Then, using the notation \eqref{eq:ola}, the assignment is given by:
\begin{align}
\begin{split}\label{eq:alc-sgn}
 &\text{if $\ol{\gamma_i} \in \wt{R}_+$, 
        then assign $+$ to $w A$-side, and $-$ to $w s_i A$-side,} \\
 &\text{if $\ol{\gamma_i} \in \wt{R}_-$, 
        then assign $-$ to $w A$-side, and $+$ to $w s_i A$-side.}
\end{split}
\end{align}
In Figure \ref{fig:fa}, we give an illustration for the case $n=2$.

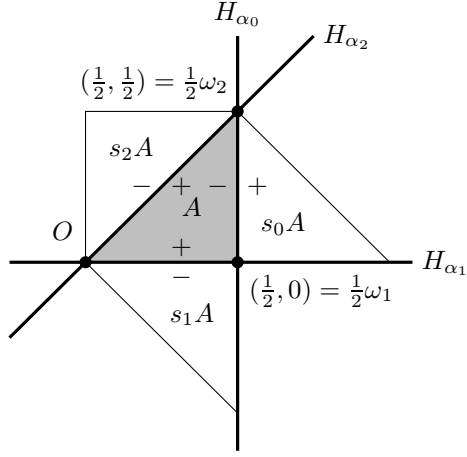
\begin{figure}[htbp]
\centering
\begin{tikzpicture}
\filldraw[fill=lightgray] (0,0)--(2,2)--(2,0)--(0,0);
 \draw (0,0)--(2,2)--(2,0)--(0,0)--(0,2)--(2,2)--(4,0)--(2,0)--(2,-2)--(0,0);
 \draw[very thick] (-1,-1.0)--(3.0,3);
 \draw[very thick] (-1, 0.0)--(4.3,0);
 \draw[very thick] ( 2,-2.5)--(2.0,3);
 \coordinate [label=center: {$H_{\alpha_0}$}]  (A) at ( 2.0, 3.3);
 \coordinate [label=right:  {$H_{\alpha_2}$}]  (B) at ( 3.0, 3.0);
 \coordinate [label=right:  {$H_{\alpha_1}$}]  (C) at ( 4.3, 0.0);
 \coordinate [label=center:            {$A$}]  (D) at ( 1.4, 0.7);
 \coordinate [label=center:        {$s_2 A$}]  (E) at ( 0.6, 1.5);
 \coordinate [label=center:        {$s_1 A$}]  (F) at ( 1.4,-0.7);
 \coordinate [label=center:        {$s_0 A$}]  (G) at ( 2.6, 0.5);
 \coordinate [label=above:             {$O$}] (O0) at (-0.3, 0.1);
 \coordinate [label=below: {$(\hf,  0)=\hf \omega_1$}] (O1) at (3.1,-0.1);
 \coordinate [label=above: {$(\hf,\hf)=\hf \omega_2$}] (O2) at (0.9, 2.1);
 \fill (0,0) circle [radius=0.08];
 \fill (2,0) circle [radius=0.08];
 \fill (2,2) circle [radius=0.08];
 \path (1, 1.0) node [right] {$+$};
 \path (1, 1.0) node [left]  {$-$};
 \path (1, 0.2) node [right] {$+$};
 \path (1,-0.2) node [right] {$-$};
 \path (2, 1.0) node [right] {$+$};
 \path (2, 1.0) node [left]  {$-$};
\end{tikzpicture}
\caption{The fundamental alcove of type $(C_2^\vee,C_2)$} \label{fig:fa}
\end{figure}

Next, we introduce alcove walks for the system $S$. 
Let $w,z \in W$ be arbitrary, and take a reduced expression 
$w=s_{i_1} s_{i_2} \dotsm s_{i_r}$ with $r:=\ell(w)$.
For a bit sequence $b=(b_1,b_2,\dotsc,b_r) \in \pr{0,1}^r$, 
consider the sequence of alcoves 
\begin{align}
 p = \bigl(p_0 := z A, \ 
           p_1 := z s_{i_1}^{b_1}A, \ 
           p_2 := z s_{i_1}^{b_1}s_{i_2}^{b_2}A, \ \dotsc, \ 
           p_r := z s_{i_1}^{b_1} \dotsm s_{i_r}^{b_r}A\bigr),
\end{align}
which is called an alcove walk with start $z$ of type $\oa{w}:=(i_1,i_2,\dotsc,i_r)$.
Note that it depends on the choice of the reduced expression of $w$,
which is indicated in the symbol $\oa{w}$.
We denote by $\Gamma(\oa{w},z)$ 
the set of all alcove walks with start $z$ of type $\oa{w}$.

\begin{eg}[Alcove walks of type $(C^\vee_2,C_2)$]\label{eg:ap}
We cite from \cite[Example 2.1.1]{YK} some illustrated examples of alcove walks
in the case $n=2$. 
Let us take $w,z \in W$ as $w = s_1 s_2 s_1 s_0$ and $z=e$.
Then we have the following two alcove walks which belong to $\Gamma(\oa{w},z)$.
\begin{align}
 p_1 := (A, A, s_2 A, s_2 s_1 A, s_2 s_1 s_0 A), \ 
 p_2 := (A, s_1 A, s_1 s_2 A, s_1 s_2 s_1 A, s_1 s_2 s_1 s_0 A).
\end{align}
We depict them in Figure \ref{fig:ap},
where the grayed alcove is the fundamental alcove $A$, 
and the number $i=0,1,2$ on each hyperplane $H$ indicates that 
$H$ belongs to the $W$-orbit of $H_{\alpha_i}$.

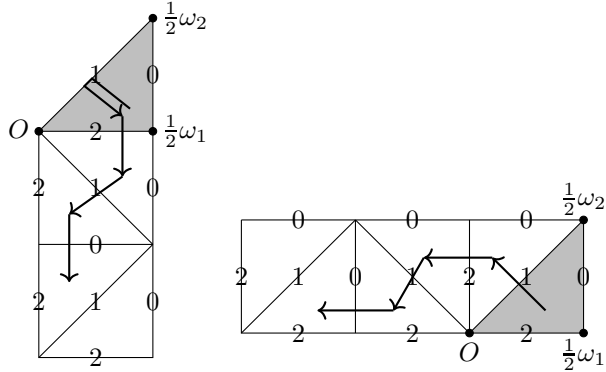
\begin{figure}[htbp]
\centering
\begin{tikzpicture}
 \filldraw[fill=lightgray] 
       (0.0, 0.0)--node[midway]{2}(1.5, 0.0)--node[midway]{0}(1.5, 1.5)
                 --node[midway]{1}(0.0, 0.0);
 \draw (0.0, 0.0)--node[midway]{2}(0.0,-1.5)--node[midway]{2}(0.0,-3.0)
                 --node[midway]{2}(1.5,-3.0)--node[midway]{0}(1.5,-1.5)
                 --node[midway]{0}(1.5, 0.0);
 \draw (0.0,-1.5)--node[midway]{0}(1.5,-1.5);
 \draw (0.0, 0.0)--node[midway]{1}(1.5,-1.5);
 \draw (1.5,-1.5)--node[midway]{1}(0,-3);
 \draw[->,thick] (1.2, 0.3)--(0.7, 0.7)--(0.6,0.6)--(1.1,0.2);
 \draw[->,thick] (1.1, 0.2)--(1.1,-0.6);
 \draw[->,thick] (1.1,-0.6)--(0.4,-1.1);
 \draw[->,thick] (0.4,-1.1)--(0.4,-2.0);
 \coordinate[label=left:            {$O$}] (A) at (0.0,0.0);
 \coordinate[label=right: {$\hf\omega_1$}] (B) at (1.5,0.0);
 \coordinate[label=right: {$\hf\omega_2$}] (C) at (1.5,1.5);
 \filldraw (A) circle [radius=0.05];
 \filldraw (B) circle [radius=0.05];
 \filldraw (C) circle [radius=0.05];
\end{tikzpicture}
\begin{tikzpicture}
 \filldraw[fill=lightgray] 
       ( 0.0,0.0)--node[midway]{2}( 1.5,0.0)--node[midway]{0}( 1.5,1.5)
                 --node[midway]{1}( 0.0,0.0);
 \draw ( 0.0,0.0)--node[midway]{2}(-1.5,0.0)--node[midway]{2}(-3.0,0.0)
                 --node[midway]{2}(-3.0,1.5);
 \draw (-3.0,1.5)--node[midway]{0}(-1.5,1.5)--node[midway]{0}( 0.0,1.5)
                 --node[midway]{0}( 1.5,1.5);
 \draw (-1.5,1.5)--node[midway]{0}(-1.5,0.0);
 \draw ( 0.0,1.5)--node[midway]{2}( 0.0,0.0);
 \draw ( 0.0,0.0)--node[midway]{1}(-1.5,1.5);
 \draw (-1.5,1.5)--node[midway]{1}(-3.0,0.0);
 \coordinate[label=below:           {$O$}] (A) at (0.0,0.0);
 \coordinate[label=below: {$\hf\omega_1$}] (B) at (1.5,0.0);
 \coordinate[label=above: {$\hf\omega_2$}] (C) at (1.5,1.5);
 \filldraw (A) circle [radius=0.05];
 \filldraw (B) circle [radius=0.05];
 \filldraw (C) circle [radius=0.05];
 \draw[->,thick] ( 1.0,0.3)--( 0.3,1);
 \draw[->,thick] ( 0.3,1.0)--(-0.6,1);
 \draw[->,thick] (-0.6,1.0)--(-1,0.3);
 \draw[->,thick] (-1.0,0.3)--(-2,0.3);
\end{tikzpicture}
\caption{Alcove walks $p_1$ and $p_2$ of type $(C^\vee_2,C_2)$}
\label{fig:ap}
\end{figure}
\end{eg}

Let us explain how we depicted alcove walks in Figure \ref{fig:ap}.
Hereafter, for an alcove walk $p \in \Gamma(\oa{w},z)$ and $k=1,2,\dotsc,r:=\ell(w)$,
the $k$-th step of $p$ means the transition $p_{k-1} \to p_k$.
Then the $k$-th bit $b_k \in \{0,1\}$ and the $k$-th step drawing correspond 
as in Table \ref{tab:bp}, where we denote by $v_{k-1}\in W$ the element giving 
the ($k-1$)-th alcove $p_{k-1}$, i.e, we have $p_{k-1}=v_{k-1}A$.
We call the $k$-th step with $b_k=1$ a crossing, 
and the step with $b_k=0$ a folding, corresponding to the drawing.

\begin{table}[htb]
\centering
\begin{tabular}{c|c|c}
$b_k$ & $1$ & $0$ \\ \hline
& crossing & folding \\ \hline
&
\begin{tikzpicture}
 \draw (1,0.8)--(1,-1);
 \draw[thick, ->] (0,0)--(2,0);
 \coordinate [label=below: {$p_{k-1}$}] (A) at (0.0,-0.2);
 \coordinate [label=below: {$p_k$    }] (B) at (2.1,-0.2);
\end{tikzpicture}
&
\begin{tikzpicture}
 \draw (1,0.8)--(1,-1);
 \draw[thick, ->] (0,0)--(1,0)--(1,-0.1)--(0,-0.1);
 \coordinate [label=below: {$p_{k-1}=p_k$}   ] (A) at (0,-0.2);
 \coordinate [label=below: {$v_{k-1}s_{i_k}A$}] (B) at (2.0,-0.2);
\end{tikzpicture} 
\end{tabular}
\caption{The correspondence of bit and alcove walk}\label{tab:bp}
\end{table}

Take $z,w \in W$ and a reduced expression $w=s_{i_1} s_{i_2} \dotsm s_{i_r}$. 
For each $p = (z A, \dotsc, z s_{i_1}^{b_1} \dotsm s_{i_r}^{b_r}A) \in \Gamma(\oa{w},z)$,
we define $e(p) \in W$ to be the element corresponding to the endpoint of $p$, i.e., 
\begin{align}\label{eq:e(p)}
 e(p) := z s_{i_1}^{b_1} s_{i_2}^{b_2} \dotsm s_{i_r}^{b_r}.
\end{align}
Also, for $k=1,2,\dotsc,r$, we define $h_k(p) \in S$ to be 
the affine root associated to the hyperplane $H$ separating $p_{k-1} = v A$ and $v s_k A$.
We also divide steps of $p\in \Gamma(\oa{w},z)$ into the four classes in 
Table \ref{tab:step} according to the signs on the sides of edges in \eqref{eq:alc-sgn}.
\begin{table}[htb]
\centering
\begin{tabular}{c|c|c|c}
positive crossing & negative crossing & positive folding & negative folding \\ \hline
\begin{tikzpicture}
 \draw (1,0.8)--(1,-1);
 \draw[thick,->] (0,0)--(2,0);
 \path (1,0.6) node [left]  {$-$};
 \path (1,0.6) node [right] {$+$};
 \coordinate [label=below:  {$p_{k-1}$}] (A) at (0.0,-0.2);
 \coordinate [label=below:  {$p_k$    }] (B) at (2.0,-0.2);
\end{tikzpicture}
&
\begin{tikzpicture}
 \draw (1,0.8)--(1,-1);
 \draw[thick,->] (0,0)--(2,0);
 \path (1,0.6) node [left ] {$+$};
 \path (1,0.6) node [right] {$-$};
 \coordinate [label=below:  {$p_{k-1}$}] (A) at (0.0,-0.2);
 \coordinate [label=below:  {$p_k    $}] (B) at (2.0,-0.2);
\end{tikzpicture}
&
\begin{tikzpicture}
 \draw (1,0.8)--(1,-1);
 \draw[thick,->] (0,0)--(1,0)--(1,-0.1)--(0,-0.1);
 \path (1,0.6) node [left ] {$+$};
 \path (1,0.6) node [right] {$-$};
 \coordinate [label=below:  {$p_{k-1}=p_k$   }] (A) at (0.0,-0.2);
 \coordinate [label=below:  {$v_{k-1}s_{i_k}A$}] (B) at (2.0,-0.2);
\end{tikzpicture}
&
\begin{tikzpicture}
 \draw (1,0.8)--(1,-1);
 \draw[thick,->] (0,0)--(1,0)--(1,-0.1)--(0,-0.1);
 \path (1,0.6) node [left ] {$-$};
 \path (1,0.6) node [right] {$+$};
 \coordinate [label=below:  {$p_{k-1}=p_k$   }] (A) at (0.0,-0.2);
 \coordinate [label=below:  {$v_{k-1}s_{i_k}A$}] (B) at (2.0,-0.2);
\end{tikzpicture}
\end{tabular}
\caption{Four classes of steps of alcove walks}
\label{tab:step}
\end{table}

Along the classification in Table \ref{tab:step}, for each 
$p \in \Gamma(\oa{w},z)$, we define $\varphi_{\pm}(p) \subset \{1,\dotsc,r\}$ by
\begin{align} 
\begin{split}
 \varphi_+(p) &:= \pr{k \mid \text{the $k$-th step of $p$ is a positive folding}}, \\
 \varphi_-(p) &:= \pr{k \mid \text{the $k$-th step of $p$ is a negative folding}}. 
\end{split}
\end{align}

\subsubsection{Ram-Yip type formula}

Recall that we introduced in \eqref{eq:Emu} the non-symmetric Koornwinder polynomial 
\begin{align}
 E_\mu(x) = E_\mu(x; q,t,t_0,t_n,u_0,u_n) \in \bK[x^{\pm1}].
\end{align}
for each $\mu \in P_{C_n}$.
We now explain Ram-Yip type formula of type $(C_n^\vee,C_n)$ in \cite{RY,OS},
which expresses the monomial expansion of $E_\mu(x)$ with the coefficients
given by summation over alcove walks.

For any $a=\alpha+r c\in S$, $\alpha \in \wt{R}$ 
(see \eqref{eq:wtR} and \eqref{eq:ola}), we define 
\begin{align}
 q^{ \sh(\alpha+r c)} := q^{-r}, \quad
 t^{\hgt(\alpha+r c)} := t^{\tbr{\rho_s^{C^\vee C},\alpha}}
 (t_0t_n)^{\tbr{\rho_l^{C^\vee C},\alpha}}, \quad 
 \rho^{C^\vee C}_s :=     \sum_{i=1}^n (n-i)\ep_i, \quad 
 \rho^{C^\vee C}_l := \thf\sum_{i=1}^n      \ep_i.
\end{align}
Let us given $v,w \in W$ and a reduced expression of $w$.
Then, for an alcove walk $p \in \Gamma(\oa{w},z)$, 
we define $\dir(p)$ and $\wgt(p)$ by the decomposition of the element $e(p) \in W$ 
(see \eqref{eq:e(p)}) along $W=t(P_{C_n})\rtimes W_0$ as 
\begin{align}
 e(p) = t(\wgt(p)) \dir(p), \quad \dir(p) \in W_0, \ \wgt(p) \in P_{C_n}.
\end{align}
Also, for $\mu \in P_{C_n}$, we denote the shortest element in the coset $t(\mu) W_0$ by
\begin{align}\label{eq:w(mu)}
 w(\mu) \in W.
\end{align} 
In the case $\mu=\ep_i$, $i=1,\dotsc,n$, the element $w(\ep_i)$ is given by
\begin{align}\label{eq:wmuCC}
 w(\ep_i) = s_{i-1}\cdots s_0.
\end{align}

\begin{fct}[{\cite[Theorem 3.1]{RY}, \cite[Theorem 3.13]{OS}}]\label{fct:RYOS}
For any $\mu \in P_{C_n}$, take a reduced expression $w(\mu)=s_{i_1}\cdots s_{i_r}$
of $w(\mu) \in W$. Then 
\begin{align}
&E_\mu(x)=\sum_{p\in\Gamma(\oa{w(\mu)},e)} f_p
 t_{\dir(p)}^{\hf}x^{\wgt(p)}, \quad 
 f_p  := 
 \prod_{k\in\varphi_+(p)}\psi_{i_k}^{+}(q^{\sh(-\beta_k)} t^{\hgt(-\beta_k))})
 \prod_{k\in\varphi_-(p)}\psi_{i_k}^{-}(q^{\sh(-\beta_k))} t^{\hgt(-\beta_k)}).
\end{align}
Here we set $\beta_k := s_{i_r} s_{i_{r-1}} \dotsm s_{i_{k+1}}(a_{i_r})$ 
for $k=1,\dotsc,r$, and $\psi_i(z)$ for $i=0,1,\dotsc,n$ are given by
\begin{align}\label{eq:psi}
\begin{split}
 \psi_j^{\pm}(z)
 :=\pm\frac{t^{-\hf}-t^{\hf}}{1-z^{\pm1}} \quad(1 \le j \le n-1),\\
 \psi_0^{\pm}(z)
 :=\pm\frac{(u_n^{-\hf}-u_n^{\hf})+z^{\pm\hf}(u_0^{-\hf}-u_0^{\hf})}{1-z^{\pm1}}, \quad 
 \psi_n^{\pm}(z)
  :=\pm\frac{(t_n^{-\hf}-t_n^{\hf})+z^{\pm\hf}(t_0^{-\hf}-t_0^{\hf})}{1-z^{\pm1}}.
\end{split}
\end{align}
\end{fct}

\subsection{Ram-Yip type $C_n$}\label{ss:RY:C}

In this subsection, we show that the Ram-Yip formula of the non-symmetric Macdonald 
polynomial of type $C_n$ in the sense of \cite{RY} can be obtained from 
the Ram-Yip type formula of type $(C_n^\vee,C_n)$ (Fact \ref{fct:RYOS}) 
by the corresponding specialization in Table \ref{tab:sp:RY}:
\begin{align}
 t_0=u_0=u_n=1.
\end{align}
See Proposition \ref{prp:RY:C} for the precise statement.

A caution on the notation is in order.
In \cite{RY}, the Ram-Yip formula for what they call type $C_n$ is derived
using the affine root system of type $C_n^\vee$ in the sense of loc.\ cit. 
As mentioned before, it turns out that both the polynomial and the root system are 
different from those in \cite{M}. For distinction, we denote by $E^{C,\RY}_\mu(x)$ and 
$S^{C^\vee,\RY}$ the polynomial and the system treated in \cite{RY}, and call them
\emph{the Macdonald polynomial of Ram-Yip type $C_n$} and 
\emph{the affine root system of Ram-Yip type $C_n^\vee$}, respectively.

\subsubsection{Affine root system of Ram-Yip type $C_n^\vee$}

We start with the explanation on the system $S^{C^\vee,\RY}$.
Let $S$ be the affine root system of type $(C_n^\vee,C_n)$ in \eqref{eq:S}.
The affine root system $S^{C^\vee,\RY}$ of Ram-Yip type $C_n^\vee$ 
is the subset of $S$ given by 
\begin{align}\label{eq:CvRY}
\begin{split}
 S^{C^\vee,\RY} := &O_1 \sqcup O_5 \\
 = &\{\pm \ep_i + r c \mid 1 \le i \le n, \, r \in \bZ\} \sqcup 
    \{\pm \ep_i \pm \ep_j + r c \mid 1 \le i < j \le n, \, r \in \bZ\},
\end{split}
\end{align}
where we used the $W$-orbits in \eqref{eq:S}.
The basis of $S^{C^\vee,\RY}$ in \cite{RY} is given by 
\begin{align}
 a_0^{C^\vee,\RY} := -(\ep_1+\ep_2)+c, \quad 
 a_j^{C^\vee,\RY} := a_j = \ep_j-\ep_{j+1} \quad (j=1,\dotsc,n-1), \quad 
 a_n^{C^\vee,\RY} := \ep_n.
\end{align}
Note that we have $a_j^{C^\vee,\RY}=a_j$ in \eqref{eq:ai}, 
but the other two roots are different from those in \eqref{eq:ai}.

Next, we turn to the extended affine Weyl group.
The refections associated to the above basis are denoted by 
\begin{align}\label{eq:Cn:sr}
 s_0^{C^\vee} := s_{a_0^{C^\vee,\RY}}, \quad 
 s_i = s_{a_i^{C^\vee,\RY}} \quad (i=1,\dotsc,n),
\end{align}
where we used $s_i \in W_0$ in \eqref{eq:W0}. 
Note that we have the common $s_n$ although $a_n^{C^\vee,\RY} \neq a_n$.
We also consider the automorphism group $\Omega^{C^\vee,\RY}$ 
of the extended Dynkin diagram of type $B_n$:
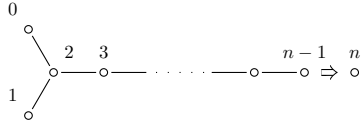
\begin{align}
 \dynkin[Kac, extended,label, label macro/.code={},labels={0,1,2,3,\ell-2,\ell-1,\ell},
         labels*={0,1,2,3,,n-1,n}]{B}{}
\end{align}
Explicitly, using the weight lattice 
$P_{B_n}=\oplus_{i=1}^n \bZ \ep_i \oplus \bZ \thf(\ep_1+\cdots+\ep_n)$ 
of type $B_n$ in \eqref{eq:PB}, we have
\begin{align}
 \Omega^{C^\vee,\RY}
 := P^\vee_{C_n}/Q^\vee_{C_n} = P_{B_n}/Q_{B_n}
  = \bbr{\pi^{C^\vee} \mid \bigl(\pi^{C^\vee}\bigr)^2=e}.
\end{align}
The generator $\pi^{C^\vee}$ flips the diagram by transposing the vertices $0 \lrt 1$.
Then, the extended affine Weyl group $W^{C^\vee,\RY}$ is defined to be the subgroup 
of $\GL_{\bR}(V)$ generated by the reflections in \eqref{eq:Cn:sr} and $\pi^{C^\vee}$. 
In other words, we have 
\begin{align}
 W^{C^\vee,\RY} := \bbr{s_0^{C^\vee},s_1,\dotsc,s_n,\pi^{C^\vee}}.
\end{align}
As an abstract group, $W^{C^\vee,\RY}$ is presented by these generators 
with the following relations.
\begin{align}\label{eq:WCrel} 
&\pi^{C^\vee}s_0^{C^\vee} = s_1\pi^{C^\vee}, & 
&{s_i}^{2} = (s_0^{C^\vee})^2=(\pi^{C^\vee})^2=e \quad (1 \le i \le n), \\
&s_0^{C^\vee} s_1 = s_1 s_0^{C^\vee}, & 
&s_i s_j = s_j s_i  \quad (\abs{i-j}>1, (i,j) \notin \pr{(0,2),(2,0)}),\\
&s_0^{C^\vee} s_2 s_0^{C^\vee}=s_2 s_0^{C^\vee} s_2, & 
&s_i s_{i+1} s_i = s_{i+1} s_i s_{i+1} \quad (1 \le i \le n-2),  \\
&s_n s_{n-1} s_n s_{n-1} = s_{n-1} s_n s_{n-1} s_n.
\end{align}
In the second line, we abusively denoted $s_0:=s_0^{C^\vee}$.
Let us write down the action of $W^{C^\vee,\RY}$ on 
$F_\bZ=P_{C_n} \oplus \thf \bZ$ in \eqref{eq:FbZ}.
\begin{align}
 s_0^{C^\vee}(\ep_i) &=
 \begin{cases}c-\ep_2 & (i = 1) \\ c-\ep_1 & (i = 2) \\ \ep_i & (i \neq 1,2)\end{cases}, & 
 s_j(\ep_i) &=
 \begin{cases}\ep_j & (i = j+1) \\ \ep_j+1 & (i = j) \\ \ep_i & (i \neq j, j+1)\end{cases}
 & (1 \le j \le n-1), \\
 s_n(\ep_i) &= 
 \begin{cases} -\ep_n & (i = n) \\ \ep_i & (i \neq n) \end{cases}, & 
 \pi^{C^\vee}(\ep_i) &= 
 \begin{cases} c-\ep_1 & (i=1) \\ \ep_i & (i \neq 1) \end{cases}. 
\end{align}
We can see from this action that $W^{C^\vee,\RY}$ preserves $S^{C^\vee,\RY} \subset S$, 
and the description $S^{C^\vee,\RY} = O_1 \sqcup O_5$ in \eqref{eq:CvRY}
is actually the decomposition into $W^{C^\vee,\RY}$-orbits.

In fact, as the following lemma shows,
the group $W^{C^\vee,\RY}$ is identical to $W$ in \eqref{eq:W}.

\begin{lem}\label{lem:WChom}
The following gives a group isomorphism $\vp^C\colon W \sto W^{C^\vee,\RY}$.
\begin{align}
 \vp^C(s_i) := s_i \quad (1 \le i \le n), \quad  \vp^C(s_0) := \pi^{C^\vee}.
\end{align}
In particular, we have the following relations of subgroups 
in $\GL_\bR(F_\bZ)$, $F_\bZ=V \oplus \bR c$.
\begin{align}
 W = W^{C^\vee,\RY} = t(P_{C_n}) \ltimes W_0.
\end{align}
\end{lem}

\begin{proof}
We regard $W$ as the group with the presentation 
$\br{s_0,s_1,\dotsc,s_n}$ in \eqref{eq:W:gr}.
Since $\vp^C(s_0 s_1 s_0) = \pi^{C^\vee} s_1 \pi^{C^\vee}=s^{C^\vee}_0$, 
we have the surjectivity of the homomorphism $\varphi^\vee$ up to well-definedness.
Thus, it is enough to show that the defining relations \eqref{eq:Wrel} of $W$ 
are mapped by $\vp^C$ to those \eqref{eq:WCrel} of $W^{C^\vee,\RY}$.
The non-trivial parts are those containing $s_0 \in W$. 
As for the fourth relation $s_0s_1s_0s_1=s_1s_0s_1s_0$ in \eqref{eq:Wrel},
the application of $\vp^C$ yields 
\begin{align}
            \vp^C(s_0 s_1 s_0 s_1) = \vp^C(s_1 s_0 s_1 s_0) \iff
 \pi^{C^\vee} s_1 \pi^{C^\vee} s_1 = s_1 \pi^{C^\vee} s_1 \pi^{C^\vee} \iff
                  s_0^{C^\vee} s_1 = s_1 s_0^{C^\vee},
\end{align}
which is in the third line of \eqref{eq:WCrel}.
The other relations are similarly checked.
\end{proof}

For later use, we write down the reduced expression of 
$t(\ep_i) \in W^{C^\vee,\RY}$ for $i=1,2,\dotsc,n$.
\begin{align}\label{eq:tepC}
\begin{split}
&t(\ep_1)=\pi^{C^\vee} s_1 \dotsm s_n s_{n-1} \dotsm s_1 \\
&t(\ep_2)=\pi^{C^\vee} s_0^{C^\vee} s_1 \dotsm s_n s_{n-1} \dotsm s_2 \\
&t(\ep_i)=\pi^{C^\vee} s_{i-1} \dotsm s_2s_0^{C^\vee}s_1 \dotsm s_n s_{n-1} \dotsm s_i
 \quad  (3 \le i \le n).
\end{split}
\end{align}

\subsubsection{Ram-Yip formula of non-symmetric Macdonald polynomials %
 of Ram-Yip type $C_n$}

Recalling the $W^{C^\vee,\RY}$-orbit decomposition $S^{C^\vee,\RY}=O_1 \sqcup O_5$
in \eqref{eq:CvRY}, we take parameters in the correspondence 
\begin{align}
 t_s^{\RY} \lrto O_1, \quad t_m^{\RY} \lrto O_5.
\end{align}
For each $\mu \in P_{C_n}$, 
we have the non-symmetric Macdonald polynomial of Ram-Yip type $C_n$, 
which is then denoted by 
\begin{align}
 E_\mu^{C,\RY}(x) = E_\mu^{C,\RY}(x;q,t_s^{\RY},t_m^{\RY}) 
 \in \bK_{C,\RY}[x^{\pm1}], \quad 
 \bK_{C,\RY} := \bQ\bigl(q^{\hf}, (t_s^{\RY})^{\hf},(t_m^{\RY})^{\hf} \bigr).
\end{align}
Below we explain the explicit formula of $E_\mu^{C,\RY}(x)$ given in \cite{RY}.

For each $a=\alpha+r c \in S^{C^\vee,\RY} \subset P_{C_n} \oplus \bR c$, 
we define $q^{\sh^C(a)}$ and $t^{\hgt^C(a)}$ by 
\begin{align}\label{eq:shhgtC}
 q^{ \sh^C(\alpha+rc)} := q^{-r}, \quad
 t^{\hgt^C(\alpha+rc)} := (t_s^{\RY})^{\tbr{\rho^C_s,\alpha}}
                          (t_m^{\RY})^{\tbr{\rho^C_m,\alpha}}, \quad 
 \rho^C_s := \sum_{i=1}^n\ep_i, \quad 
 \rho^C_m=\sum_{i=1}^n(n-i)\ep_i.
\end{align}

We also denote the fundamental alcove of $S^{C^\vee,\RY}$ by 
\begin{align} 
 A^{C^\vee,\RY} := \pr{x\in V \mid a_i^{C^\vee,\RY}(x) \ge 0, \  i=0,1,\dotsc,n}.
\end{align}
Then we have $A^{C^\vee,\RY} = A \cup s_0 A $, where $A$ is the fundamental alcove
\eqref{eq:FA} and $s_0$ is the $0$-th reflection associated to 
$a_0 = -2\ep_1+c \in S$ \eqref{eq:RY:a0}, both of type $(C^\vee_n,C_n)$.
Note that $a_0 \neq a_0^{C^\vee,\RY}$, so that the corresponding hyperplanes 
and reflections are different. See Figure \ref{fig:faC} for the case $n=2$.

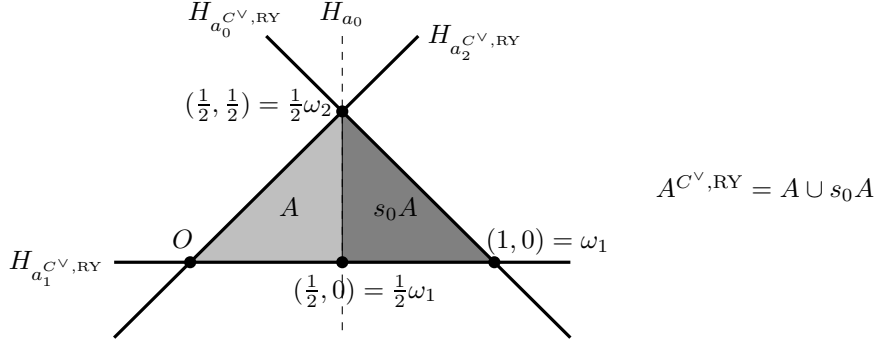
\begin{figure}[htbp]
\centering
\begin{tikzpicture}
\filldraw[fill=gray] (0,0)--(2,2)--(4,0);
\filldraw[fill=lightgray] (0,0)--(2,2)--(2,0); 
 \draw (0,0)--(2,2)--(4,0);
 \draw[very thick] (-1,-1.0)--(3.0,3);
 \draw[very thick] (-1, 0.0)--(5.0,0);
 \draw[very thick] (1,3)--(5.0,-1);
 \draw[dashed] (2,3)--(2,-1);
 \coordinate [label=center: {$H_{a_0^{C^\vee,\RY}}$}]  (A) at ( 0.6, 3.3);
 \coordinate [label=right:  {$H_{a_2^{C^\vee,\RY}}$}]  (B) at ( 3.0, 3.0);
 \coordinate [label=left:   {$H_{a_1^{C^\vee,\RY}}$}]  (C) at (-1.0, 0.0);
 \coordinate [label=center: {$H_{a_0}$}]  (C) at ( 2, 3.3);
 \coordinate [label=center: {$A$}]    (D) at ( 1.3, 0.7);
 \coordinate [label=center: {$s_0A$}] (E) at ( 2.7, 0.7);
 \coordinate [label=above:  {$O$}]    (O) at (-0.1, 0.0);
 \coordinate [label=below:  {$(\hf,  0)=\hf \omega_1$}] (O1) at (2.3,-0.1);
 \coordinate [label=above:  {$(\hf,\hf)=\hf \omega_2$}] (O2) at (0.9, 1.8);
 \coordinate [label=above:  {$(1,0)=\omega_1$}] (O3) at (4.7, 0.0);
 \fill (0,0) circle [radius=0.08];
 \fill (2,0) circle [radius=0.08];
 \fill (2,2) circle [radius=0.08];
 \fill (4,0) circle [radius=0.08];
 \node (M) at (6,1) [right]{$A^{C^\vee,\RY} = A \cup s_0 A$};
\end{tikzpicture}
\caption{The fundamental alcove $A^{C^\vee,\RY}$ of Ram-Yip type $C_2$} \label{fig:faC}
\end{figure}

Finally, for each $\mu \in P_{C_n}$, we denote the shortest element 
in the coset $t(\mu) W_0$ by
\begin{align}\label{eq:wC(mu)}
 w_C(\mu) \in W^{C^\vee,\RY}. 
\end{align}
Finally, we denoted by $\Gamma_C(\oa{w},z)$ 
the set of all alcove walks with start $z \in W^{C^\vee,\RY}$ of type $\oa{w}$.

\begin{fct}[{\cite[Theorem 3.1]{RY}}]\label{fct:CRY}
Let $\mu \in P_{C_n}$ be arbitrary, and take a reduced expression
$w_C(\mu) = (\pi^{C^\vee})^k s_{i_1} \dotsm s_{i_r}$ with $k \in \pr{0,1}$,
using the abbreviated symbols in \eqref{eq:WCrel}. Then, we have
\begin{align}
&E^{C,\RY}_\mu(x)=\sum_{p\in\Gamma_C(\oa{w_(\mu)},e)} f^C_p
 t_{\dir(p)}^{\hf}x^{\wgt(p)}, \\
&f^C_p  := 
 \prod_{k\in\varphi_+(p)}(\psi^C_{i_k})^{+}(q^{\sh^C(-\beta_k)} t^{\hgt^C(-\beta_k)})
 \prod_{k\in\varphi_-(p)}(\psi^C_{i_k})^{-}(q^{\sh^C(-\beta_k)} t^{\hgt^C(-\beta_k)}),
\end{align}
where $\beta_k := s_{i_r} s_{i_{r-1}} \dotsm s_{i_{k+1}}(a_{i_r}^{C^\vee,\RY})$ 
for $k=1,2,\dotsc,r$, and $(\psi^C_i)^{\pm}(z)$ for $i=0,1,\dotsc,n$ is given by 
\begin{align}\label{eq:CRY:psi}
 (\psi^C_i)^{\pm}(z) := \pm\frac{(t_m^{\RY})^{-\hf}-(t_m^{\RY})^{\hf}}{1-z^{\pm1}}
  \quad (0 \le i \le n-1), \quad
 (\psi^C_n)^{\pm}(z) := \pm\frac{(t_s^{\RY})^{-\hf}-(t_s^{\RY})^{\hf}}{1-z^{\pm1}}.
\end{align}
\end{fct}

\subsubsection{Specialization to type $C_n$}

In this part, we check that the specialization $t_0=u_0=u_n=1$ of the Ram-Yip type formula
for the non-symmetric Koornwinder polynomial  $E_\mu(x)$ (Fact \ref{fct:RYOS}) is equal to
the Ram-Yip formula for the non-symmetric Macdonald polynomial $E_\mu^{C,\RY}(x)$ 
of Ram-Yip type $C_n$ (Fact \ref{fct:CRY}). 
Using \eqref{eq:Emu}, we denote the specialized non-symmetric Koornwinder polynomial by
\begin{align}\label{eq:EspC}
 E_\mu^{\tsp,C}(x) = E_\mu^{\tsp,C}(x;q,t,t_n) := E_\mu(x;q,t,1,t_n,1,1).
\end{align}

We denote by 
\begin{align}
 \Gamma_0(\oa{w(\mu)},e) \subset \Gamma(\oa{w(\mu)},e)
\end{align}
the subset consisting of alcove walks without folding by $s_0$.
We first show that under the specialization $t_0=u_0=u_n=1$, 
the summation over $\Gamma(\oa{w(\mu)},e)$ in Fact \ref{fct:RYOS} 
reduces to that over $\Gamma_0(\oa{w(\mu)},e)$.

\begin{lem}\label{lem:SpCRY}
Let $\mu \in P_{C_n}$ be arbitrary, and 
take a reduced expression $w(\mu)=s_{i_1} \dotsm s_{i_r}$ for the element 
$w(\mu) \in W$ given in \eqref{eq:w(mu)}. Then 
\begin{align}
&E^{\tsp,C}_\mu(x)=\sum_{p \in \Gamma_0(\oa{w_(\mu)},e)} f_p
 t_{\dir(p)}^{\hf}x^{\wgt(p)}, \\
&f_p  := 
 \prod_{k\in\varphi_+(p)}(\psi^{\tsp,C}_{i_k})^{+}(q^{\sh(-\beta_k)}  t^{\hgt(-\beta_k)})
 \prod_{k\in\varphi_-(p)}(\psi^{\tsp,C}_{i_k})^{-}(q^{\sh(-\beta_k))} t^{\hgt(-\beta_k)}),
\end{align}
where we used 
\begin{align}\label{eq:CRY:psi^sp}
 (\psi^{\tsp,C}_i)^{\pm}(z) := \pm \frac{t^{-\hf}-t^{\hf}}{1-z^{\pm1}} 
 \quad (i=1,\dotsc,n-1), \quad 
 (\psi^{\tsp,C}_n)^{\pm}(z) := \pm \frac{t_n^{-\hf}-t_n^{\hf}}{1-z^{\pm1}}.
\end{align}
\end{lem}

\begin{proof}
The specialization $u_0=u_n=1$ yields $\psi^{\pm}_0(z)=0$ by \eqref{eq:psi}.
Thus, no folding step by $s_0$ appear in the summation in Fact \ref{fct:RYOS}. 
Also, a direct calculation shows that under $t_0=1$,
$\psi_i^{\pm} (z)$ is equal to $(\psi^{\tsp,C}_i)^{\pm}(z)$ for $i=1,\dotsc,n$.
\end{proof}

Comparing \eqref{eq:CRY:psi} and \eqref{eq:CRY:psi^sp}, we have
\begin{align}\label{eq:CRY:pp}
 (\psi^{\tsp,C}_i)^{\pm}(z)\Bigr|_{t  =t_m^{\RY}} = (\psi^C_i)^{\pm}(z), \quad 
 (\psi^{\tsp,C}_n)^{\pm}(z)\Bigr|_{t_n=t_s^{\RY}} = (\psi^C_n)^{\pm}(z).
\end{align}
Hence, to check the identification of $E^{C,\RY}_\mu(x)$ with $E^{\tsp,C}_\mu(x)$,
it is enough to construct a bijection 
\begin{align}
 \Gamma_0(\oa{w(\mu)},e) \lto \Gamma_C(\oa{w_C(\mu)},e)
\end{align}
between the sets of alcove walks.

\begin{lem}\label{lem:isom}
For any $\mu \in P_{C_n}$, take a reduced expression $w(\mu)=s_{i_1}\cdots s_{i_{\ell}}$
of the element $w(\mu) \in W$ in \eqref{eq:w(mu)}, and set 
\begin{align}
&I := \pr{r \in \pr{1,2,\dotsc,\ell} \mid i_r \neq 0} = \pr{k_1<k_2<\dotsb<k_s} 
 \quad (s \le \ell), \\
&J := \bpr{(b_1,b_2,\dotsc,b_\ell) \in \pr{0,1}^\ell \mid  b_i=1 \ (i \notin I) }.
\end{align}
Also, define $\theta^C\colon J \to \pr{0,1}^s$ by 
\begin{align}
 J \ni (b_1,b_2,\dotsc,b_\ell) \lmto (b_{k_1},b_{k_2},\dotsc,b_{k_s}) \in \pr{0,1}^s.
\end{align}
Then the following statements hold.
\begin{enumerate}[nosep]
\item 
The length of $w_C(\mu) \in W(C^{\vee,\RY})$ is $\abs{I}=s$, and 
we can write $w_C(\mu)$ by 
\begin{align}
 w_C(\mu) = \begin{cases} s_{j_1} s_{j_2} \dotsm s_{j_s} & (s    \in 2 \bN) \\ 
             \pi^{C^\vee} s_{j_1} s_{j_2} \dotsm s_{j_s} & (s \notin 2 \bN) \end{cases}
\end{align}
with some $j_r$'s, where we used the abbreviation in \eqref{eq:WCrel}.

\item
The map $\theta^C\colon J \to \pr{0,1}^s$ induces a bijection 
\begin{gather}
 \wt{\theta^C}\colon \Gamma_0(\oa{w(\mu)},e) \lto \Gamma_C(\oa{w_C(\mu)},e), \\
 p=(A,s_{i_1}^{b_1}A, \dotsc, s_{i_1}^{b_1}\cdots s_{i_{\ell}}^{b_\ell}A) \lmto 
 \begin{cases}
 (A_C, s_{j_1}^{b_{k_1}}A_C,\dotsc,s_{j_1}^{b_{k_1}} \dotsm s_{j_s}^{b_{k_s}}A_C) 
 & (s \in 2\bN) \\
 (\pi^{C^\vee} A_C, \pi^{C^\vee} s_{j_1}^{b_{k_1}}A_C,\dotsc,
  \pi^{C^\vee} s_{j_1}^{b_{k_1}} \dotsm s_{j_s}^{b_{k_s}}A_C) 
 &(s \notin 2\bN)
\end{cases}.
\end{gather}

\item
For any $p \in \Gamma_0(\oa{w(\mu)},e)$, we have 
\begin{align}
 \wgt(p) = \wgt\bigl(\wt{\theta^C}(p)\bigr), \quad 
 \dir(p) = \dir\bigl(\wt{\theta^C}(p)\bigr).
\end{align}
\end{enumerate}
\end{lem}

\begin{proof}
\begin{enumerate}[nosep]
\item 
It is enough to show $\vp^C(w(\mu))=w_C(\mu)$ for any $\mu \in P_{C_n}$.
First, we can see $\vp^C(w(\ep_i))=w_C(\ep_i)$ by the comparison between 
the reduced expressions \eqref{eq:tep} and \eqref{eq:tepC}.
Since $\vp^C$ is a group isomorphism by Lemma \ref{lem:WChom},
we see that $\vp^C(w(\mu))=w_C(\mu)$ for any $\mu \in P_{C_n}$.

\item
It is an immediate consequence of the item (1) and the bijectivity of $\theta$.

\item
We want to show that for any $p \in \Gamma_0(\oa{w(\mu)},e)$, expressing 
$e(p)=t(\wgt(p))\dir(p)$, $\wgt(p)\in P_{C_n}$, $\dir(p)\in W_0$, we would have 
$\varphi^C(e(p))=t\bigl(\wgt(\wt{\theta^C}(p))\bigr) \dir\bigl(\wt{\theta^C}(p)\bigr)$.
For any $i=1,2,\dotsc,n$, we have $\vp^C(t(\ep_i)) = t(\ep_i)$ 
by the comparison of \eqref{eq:tep} with \eqref{eq:tepC}.
Thus we have $t(\wgt(p))=t(\wgt(\wt{\theta^C}(p)))$ for any $p$, 
which means $\wgt(p)=\wgt(\wt{\theta^C}(p))$.
On the other hand, since $\rst{\vp^C}{W_0} = \id_{W_0}$, 
we have $\dir(p)=\dir(\wt{\theta^C}(p))$ for any $p$.
Thus the statement is proved.
\end{enumerate}
\end{proof}

Combining this lemma with \eqref{eq:CRY:pp}, we obtain the desired identification
\begin{align}
 E_\mu^{\tsp,C}(x;q,t=t_m^{\RY},t_n=t_s^{\RY})=E_\mu^{C,\RY}(x;q,t_s^{\RY},t_m^{\RY}).
\end{align}
The definition \eqref{eq:EspC} of $E_\mu^{\tsp,C}(x)$ yields:

\begin{prp}\label{prp:RY:C}
For any $\mu \in P_{C_n}$, we have 
\begin{align}
 E_\mu(x;q,t_m^{\RY},1,t_s^{\RY},1,1) = E_\mu^{C,\RY}(x;q,t_s^{\RY},t_m^{\RY}).
\end{align}
\end{prp}

Comparing this result with the specialization Table \ref{tab:sp},
we see that it corresponds to type $B_n^\vee$.
Thus, the Macdonald polynomial of Ram-Yip type $C_n$ is 
the Macdonald polynomial of type $B_n^\vee$ in the sense of Definition \ref{dfn:type}.

\subsection{Ram-Yip type $B_n$}\label{ss:RY:B}

The Ram-Yip formula of non-symmetric Macdonald polynomial of type $B_n$ 
is derived in \cite{RY} using the affine root system of type $B^\vee_n$ 
in the sense of loc.\ cit.
In this subsection, we give a similar argument as in the previous \S \ref{ss:RY:C} to 
type $B_n$, and show that under the specialization
\begin{align}
 t_n=u_n, \quad t_0=u_0=1
\end{align}
we can recover the non-symmetric Macdonald polynomial of type $B_n$ 
in the sense of \cite{RY} from the non-symmetric Koornwinder polynomial.

We will use similar terminologies on the affine root system and 
the non-symmetric Macdonald polynomials as in \S \ref{ss:RY:C}.
We denote by $S^{B^\vee,\RY}$ and $E^{B,\RY}_\mu(x)$ 
those considered in \cite{RY} for type $B$, 
and call them \emph{the affine root system of Ram-Yip type $B_n^\vee$} and 
\emph{the Macdonald polynomial of Ram-Yip type $B_n$}, respectively.

\subsubsection{Affine root system of Ram-Yip type $B_n^\vee$}

Using the symbols in \eqref{eq:S}, 
the affine root system $S^{B^\vee,\RY}$ of Ram-Yip type is given by 
\begin{align}\label{eq:BvRY}
\begin{split}
 S^{B^\vee,\RY}
:= &(O_2 \sqcup O_4) \sqcup O_5 \\
 = &\{\pm 2\ep_i+r \mid 1 \le i \le n, \, r \in \bZ\} \sqcup 
    \{\pm \ep_i \pm \ep_j + r \mid 1 \le i < j \le n, \, r \in \bZ\}.
\end{split}
\end{align}
The choice of the basis in \cite{RY} is given by 
\begin{align}
 a_0^{B^\vee,\RY} := a_0 = -2\ep_1+c, \quad  
 a_j^{B^\vee,\RY} := a_j =  \ep_{j}-\ep_{j+1} \quad  (j=1,\dotsc,n-1), \quad
 a_n^{B^\vee,\RY} := a_n =  2\ep_n,
\end{align}
where $a_i$'s are in \eqref{eq:ai}. Thus, the associated reflections are 
$s_{a^{B^\vee}_i}=s_i$ in \eqref{eq:W0} and \eqref{eq:s0}.

We turn to the explanation of the extended affine Weyl group.
Let $\Omega_{B^\vee}$ be the automorphism group of the extended Dynkin diagram 
of type $C_n$:
\[
 \dynkin[Kac, extended,label,label macro/.code={},labels={0,1,2,\ell-2,\ell-1,\ell},
         labels*={1,2,2,2,2,1}]{C}{}
\]
Explicitly, we have 
\begin{align}
 \Omega_{B^\vee} := P^\vee_{B_n}/Q^\vee_{B_n} = P_{C_n}/Q_{C_n}
 = \tbr{\pi^{B^\vee} \mid (\pi^{B^\vee})^2=e}.
\end{align}
Then, the extended affine Weyl group $W^{B^\vee,\RY}$ is the subgroup 
of $\GL_{\bR}(V)$, $V=\oplus_{i=1}^n \bR \ep_i$ given by 
\begin{align}
 W^{B^\vee,\RY} := \bbr{s_0,s_1,\dotsc,s_n,\pi^{B^\vee}}.
\end{align}
As an abstract group, $W^{B^\vee,\RY}$ has a presentation with these generators
and the following relations.
\begin{align}
\begin{split}
 s_i^2 = 1 & \qquad (i=0,\dotsc,n), \\
 s_i s_j = s_j s_i & \qquad (\abs{i-j}>1), \\
 s_i s_{i+1} s_i = s_{i+1} s_i s_{i+1} & \qquad (i=1,\dotsc,n-2), \\
 s_i s_{i+1} s_i s_{i+1} = s_{i+1} s_i s_{i+1} s_i & \qquad (i=0,n-1), \\
 \pi^{B^\vee} s_i= s_{n-i+1} \pi^{B^\vee} & \qquad (i=0,1,\dotsc,n).
\end{split}
\end{align}
Let us write down the action of $W^{B^\vee,\RY}$ on 
$F_\bZ=P_{C_n} \oplus \thf \bZ c$ \eqref{eq:FbZ}.
\begin{align}
&s_0(\ep_i)=\begin{cases}c-\ep_1& (i=  1)\\ \ep_i  & (i \neq 1)\end{cases}, & 
&s_j(\ep_i)=\begin{cases}  \ep_j& (i=j+1)\\ \ep_j+1& (i   =  j) \\
                           \ep_i& (i\neq j,j+1)                \end{cases}
             \quad (j=1,\dotsc,n-1), \\
&s_n(\ep_i)=\begin{cases} -\ep_n& (i=  n)\\ \ep_i  & (i \neq n) \end{cases}, & 
&\pi^{B^\vee}(\ep_i)=\hf c-\ep_{n-i+1}
\end{align}
We can see from this action that $W^{B^\vee,\RY}$ acts on $S^{B^\vee,\RY}$, 
and the description $S^{B^\vee,\RY} = O_1 \sqcup O_5$ in \eqref{eq:BvRY}
is actually the decomposition into $W^{B^\vee,\RY}$-orbits.

The group $W^{B^\vee,\RY}$ also has the following descriptions.
\begin{align}\label{eq:WBvRY}
 W^{B^\vee,\RY} = \Omega_{B^\vee} \rtimes W = t(P_{B_n}) \rtimes W_0, \quad 
 P_{B_n} := \bZ\ep_1 \oplus \dotsb \oplus \bZ\ep_n \oplus \bZ\thf(\ep_1+\dots+\ep_n),
\end{align}
where we used $t$ in \eqref{eq:t(v)}.
For later use, we write down reduced expressions of $t(\ep_i)$'s.
\begin{align}\label{eq:tepB}
\begin{split}
&t(\ep_i)=s_{i-1} \dotsm s_1 s_0 s_1 \dotsm s_n s_{n-1} \dotsm s_i 
 \quad (i=1,2,\dotsc,n), \\
&t\bigl(\thf(\ep_1+\dots+\ep_n)\bigr) = 
 \pi^{B^\vee} (s_n \dotsm s_1) \dotsm (s_n s_{n-1})s_n.
\end{split}
\end{align}

\subsubsection{Ram-Yip formula of non-symmetric Macdonald polynomial of type $B_n$}

Next we consider the parameters for Macdonald polynomials.
Recalling the $W^{B^\vee,\RY}$-orbit decomposition 
$S^{B^\vee,\RY} = O_5 \sqcup (O_2 \sqcup O_4)$ in \eqref{eq:BvRY},
we take parameters $t_m^{\RY}$ and $t_l^{\RY}$ in the correspondence 
\begin{align}
 t_m^{\RY} \lrto O_5,\quad t_l^{\RY} \lrto O_2 \sqcup O_4.
\end{align}
We have the non-symmetric Macdonald polynomial of Ram-Yip type $B_n$
for $\mu \in P_{B_n}$ in \eqref{eq:PB}, which is then denoted by 
\begin{align}
 E_\mu^{B,\RY}(x) = E_\mu^{B,\RY}(x;q,t_m^{\RY},t_l^{\RY}) 
 \in \bK_{B,\RY}[x^{\pm1}], \quad 
 \bK_{B,\RY} := \bQ\bigl(q^{\hf}, (t_m^{\RY})^{\hf},(t_l^{\RY})^{\hf} \bigr).
\end{align}
For each $a=\alpha+r c \in S^{B^\vee,\RY}$, 
we define $q^{\sh^B(a)}$ and $t^{\hgt^B(a)}$ by 
\begin{align}\label{eq:shhgtB}
 q^{ \sh^B(\alpha+rc)}:=q^{-r}, \quad
 t^{\hgt^B(\alpha+rc)}:=t_m^{\tbr{\rho^B_m,\alpha}} t_l^{\tbr{\rho^B_l,\alpha}}, \quad
 \rho^B_m :=     \sum_{i=1}^n(n-i)\ep_i, \quad 
 \rho^B_l := \thf\sum_{i=1}^n     \ep_i.
\end{align}
We denote the fundamental alcove of Ram-Yip type $B_n^\vee$ by 
\begin{align} 
 A^{B^\vee,\RY} := \bpr{x \in V \mid a_i^{B^\vee}(x) \ge 0, \ i=0,1,\dotsc,n}.
\end{align}
See Figure \ref{fig:faB} for the case $n=3$.
We have $A^{B^\vee,\RY}=A$ in \eqref{eq:FA}.

\begin{figure}[htbp]
\centering
\begin{tikzpicture}
 \draw(0,0)--(2.5,0)--(1.2,2)--(0,0);
 \draw[dashed] (0,0)--(1.1,0.7);
 \draw[dashed] (2.5,0)--(1.1,0.7);
 \draw[dashed] (1.2,2)--(1.1,0.7);
 \fill (0,0) circle (2pt);
 \fill (2.5,0) circle (2pt);
 \fill (1.2,2) circle (2pt);
 \fill (1.1,0.7) circle (2pt);
 \fill[lightgray, opacity=0.2](0,0)--(2.5,0)--(1.2,2)--(0,0);
 \node [below] at (0,0) {$\hf\ep_{1}$};
 \node [below] at (2.5,0) {$\hf\ep_{2}$};
 \node [left] at (1,2) {$\hf\ep_{3}$};
 \node [left] at (1.1,0.7) {$O$};
 \coordinate [label=right: {$A^{B^\vee,\RY}$}] (A) at (2.5, 1);
\end{tikzpicture}
\caption{The fundamental alcove of Ram-Yip type $B_3^\vee$} \label{fig:faB}
\end{figure}
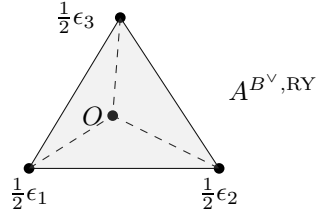

We also denote by $w_B(\mu) \in W^{B^\vee,\RY}$ 
the shortest element of the coset $t(\mu) W^{B^\vee,\RY}$.
Finally, $\Gamma_B(\oa{w},z)$ denotes the set of alcove walks 
with start $z \in W^{B^\vee,\RY}$ of type $\oa{w}$,

\begin{fct}[{\cite[Theorem 3.1]{RY}}]\label{fct:BRY}
Let $\mu \in P_{B_n}$ be arbitrary, and take a reduced expression 
$w_B(\mu)=(\pi^{B^\vee})^k s_{i_1} s_{i_2} \dotsm s_{i_r}$, $k \in \pr{0,1}$.
Then we have 
\begin{align}
&E^{B,\RY}_\mu(x) = \sum_{p\in\Gamma_B(\oa{w_(\mu)},e)} f^B_p
              t_{\dir(p)}^{\hf} x^{\wgt(p)}, \\
&f^B_p := 
 \prod_{k \in \vp_+(p)}(\psi^B_{i_k})^+(q^{\sh^B(-\beta_k)} t^{\hgt^B(-\beta_k)})
 \prod_{k \in \vp_-(p)}(\psi^B_{i_k})^-(q^{\sh^B(-\beta_k)} t^{\hgt^B(-\beta_k)}),
\end{align}
where $\beta_k := s_{i_r} \dotsm s_{i_{k+1}}(a^B_{i_r})$ for $k=1,2,\dotsc,r$, 
and $(\psi^B_i)^{\pm}(z)$ for $i=0,1,\dotsc,n$ is given by
\begin{align}\label{eq:BRY}
\begin{split}
&(\psi^B_j)^{\pm}(z) := \pm\frac{(t_m^{\RY})^{-\hf}-(t_m^{\RY})^{\hf}}{1-z^{\pm1}}
 \quad (1 \le j \le n-1), \\ 
&(\psi^B_i)^{\pm}(z) := \pm\frac{(t_l^{\RY})^{-\hf}-(t_l^{\RY})^{\hf}}{1-z^{\pm1}}
 \quad (i=0,n).
\end{split}
\end{align}
\end{fct}

\subsubsection{Specialization to type $B_n$}

In this part, we check that the specialization $t_n=u_n$, $t_0=u_0=1$ of $E_\mu(x)$ 
in Fact \ref{fct:RYOS} is equal to $E_\mu^{B,\RY}(x)$ in Fact \ref{fct:BRY}. 
Using \eqref{eq:Emu}, we denote the specialized non-symmetric Koornwinder polynomial by
\begin{align}\label{eq:EspB}
 E_\mu^{\tsp,B}(x) = E_\mu^{\tsp,B}(x;q,t,t_n) := E_\mu(x;q,t,1,t_n,1,t_n).
\end{align}

\begin{lem}\label{lem:WBhom}
The map $s_i \mapsto s_i$ ($i=0,\dotsc,n$) defines an injective group homomorphism 
$W \inj W^{B^\vee,\RY}$. 
\end{lem}

\begin{proof}
Obvious from the structure 
$W^{B^\vee,\RY} = \Omega_{B^\vee} \rtimes W$ in \eqref{eq:WBvRY}.
\end{proof}

\begin{lem}\label{lem:SpBRY}
For any $\mu \in P_{C_n}$, take a reduced expression 
$w(\mu)=s_{i_1} s_{i_2} \dotsm s_{i_r}$ of the element 
$w(\mu) \in W$ in \eqref{eq:w(mu)}. Then, we have
\begin{align}
&E^{\tsp,B}_\mu(x)=\sum_{p\in\Gamma(\oa{w_(\mu)},e)} f_p
 t_{\dir(p)}^{\hf}x^{\wgt(p)}, \\
&f_p  := 
 \prod_{k \in \vp_+(p)}(\psi^{\tsp,B}_{i_k})^+(q^{\sh(-\beta_k)} t^{\hgt(-\beta_k)})
 \prod_{k \in \vp_-(p)}(\psi^{\tsp,B}_{i_k})^-(q^{\sh(-\beta_k)} t^{\hgt(-\beta_k)}),
\end{align}
where $\beta_k := s_{i_r} s_{i_{r-1}} \dotsm s_{i_{k+1}}(a_{i_r})$ for $k=1,2,\dotsc,r$, 
and $(\psi^{\tsp,B}_i)^{\pm}(z)$ for $i=0,1,\dotsc,n$ is given by 
\begin{align}\label{eq:BRY:sp}
 (\psi^{\tsp,B}_j)^{\pm}(z) := \pm\frac{t^{-\hf}-t^{\hf}}{1-z^{\pm1}}
  \quad (1 \le j \le n-1), \quad 
 (\psi^{\tsp,B}_i)^{\pm}(z) := \pm\frac{t_n^{-\hf}-t_n^{\hf}}{1-z^{\pm1}}
  \quad (i=0,n).
\end{align}
\end{lem}

\begin{proof}
A direct calculation with $t_n=u_n$ and $t_0=u_0=1$ yields the result.
\end{proof}
 
Since $P_{C_n} \subset P_{B_n}$, we have:

\begin{prp}\label{prp:RY:B}
For any $\mu \in P_{C_n}$, the following equality holds.
\begin{align}\label{eq:RY:B}
 E_\mu(x;q,t_m^{\RY},1,t_l^{\RY},1,t_l^{\RY}) = E_\mu^{B,\RY}(x;q,t_m^{\RY},t_l^{\RY}).
\end{align}
\end{prp}

\begin{proof}
By \eqref{eq:EspB}, it is enough to show 
$E_\mu^{\tsp,B}(x;q,t_m^{\RY},t_l^{\RY}) = E^B_\mu(x;q,t_m^{\RY},t_l^{\RY})$.
Comparing \eqref{eq:BRY} and \eqref{eq:BRY:sp}, we have
\begin{align}
 \rst{(\psi^{\tsp,B}_j)^{\pm}(z)}{t  =t_m^{\RY}} = (\psi^B_j)^{\pm}(z), \quad 
 \rst{(\psi^{\tsp,B}_i)^{\pm}(z)}{t_n=t_l^{\RY}} = (\psi^B_i)^{\pm}(z).
\end{align}
The embedding $W \inj W^{B^\vee,\RY}$ in Lemma \ref{lem:WBhom} implies that 
we have $\Gamma(\oa{w(\mu)},e)=\Gamma_B(\oa{w(\mu)},e)$
for any $\mu \in P_{B_n} \cap P_{C_n}=P_{C_n}$. 
Then, the result follows from Lemma \ref{lem:SpBRY}.
\end{proof}

Comparing this result with the specialization Table \ref{tab:sp},
we see that the specialization \eqref{eq:RY:B} corresponds to type $B_n$.
Thus, the Macdonald polynomial of Ram-Yip type $B_n$ is 
the Macdonald polynomial of type $B_n$ in the sense of Definition \ref{dfn:type}.

\subsection{Type $D_n$}\label{ss:RY:D}

By Proposition \ref{prp:D}, we know that the specialization
\begin{align}
 t_n=u_n=t_0=u_0=1
\end{align}
yields the non-symmetric Macdonald polynomial $E^D_\mu(x)$ of type $D_n$.
In this subsection, we reprove it by using the Ram-Yip formula of type $D_n$,
in which case there is no discrepancy between \cite{M} and \cite{RY},
so we use our default notation for the affine root system and 
the non-symmetric Macdonald polynomials based on \cite{M}.

\subsubsection{Ram-Yip affine root system of type $D$}

Recall the affine root system $S^D$ of type $D_n$ given in \eqref{eq:S^D}:
\begin{align}
 S^D := O_5 = \{\pm \ep_i \pm \ep_j + r \mid 1 \le i < j \le n, \, r \in \bZ\}.
\end{align}
A basis given by 
\begin{align}
 a_0^D := -\ep_1-\ep_2+c, \quad 
 a_j^D := a_j =  \ep_j-\ep_{j+1} \quad (1 \le j \le n-1), \quad 
 a_n^D := \ep_{n-1}+\ep_n.
\end{align}
Denoting $s^D_n := s_{a_n^D}$, the finite Weyl group is given by 
$W_0^D := \bbr{s_1,\dotsc,s_{n-1},s^D_n} \simeq \pr{\pm1}^{n-1} \rtimes \frS_n$.
Also, recall the weight lattice $P_{D_n}$ in \eqref{eq:PD}:
\begin{align}
 P_{D_n} := \bZ\ep_1 \oplus \dotsb \oplus \bZ\ep_n \oplus \bZ\thf(\ep_1+\dotsb+\ep_n)
\end{align}
and the extended affine Weyl group $W^D  = W_0^D \ltimes t(P_{D_n})$ in \eqref{eq:WD}.
The group $W^D$ has another description:
\begin{align}\label{eq:WDstr}
 W^D = \br{s_0^D,s_1,\dotsc,s_{n-1},s^D_n,\pi^D_1,\pi_{n-1}^{D},\pi^D_n}.
\end{align}
Here $\pi_1^D$, $\pi_{n-1}^D$ and $\pi_n^D$ denotes the generators
of the automorphic group 
\begin{align}
 \Omega_D := P_{D_n}/Q_{D_n}=\br{\pi^D_0=e,\pi_1^{D},\pi_{n-1}^D,\pi_n^D }
\end{align}
of the extended Dynkin diagram of type $D_n$:
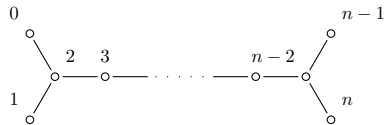
\[
 \dynkin[Kac, extended,label,label macro/.code={},labels={0,1,2,\ell-2,\ell-1,\ell},
         labels*={0,1,2,3,,n-2,n-1,n}]{D}{}
\]
As an abstract group, $W^D$ is presented by the generators \eqref{eq:WDstr} 
and the following relations.
\begin{align}\label{eq:WDrel} 
&s_i^2 = (s_0^D)^2=e, \\
&s_0^D s_1 = s_1 s_0^D, &
&s_{n-1} s_n^D = s_n^D s_{n-1}, & 
&s_i s_j = s_j s_i \quad (\abs{i-j}>1), \\
&s^D_0 s_2 s^D_0 = s_2 s_0^D s_2, & 
&s^D_n s_{n-2} s^D_n = s_{n-2} s_n^D s_{n-2}, & 
&s_i s_{i+1} s_i = s_{i+1} s_i s_{i+1} \quad (i=1,\dotsc,n-2), \\
&\pi^D_1 s_0 = s_1 \pi^D_1, & 
&\pi^D_{n-1} s^D_{0}=s_{n-1} \pi^D_{n-1}, &
&\pi_n^D s_0^D = s_n^D \pi_n^D, \\
&\pi_{n-1}^D s_1 = s_n \pi_{n-1}^D, & 
&\pi_n^D s_1 = s_n \pi_n^D, & 
&\pi_{n-1}^D s_i = s_{n-i} \pi_{n-1}^D \quad (i=2,\dotsc,n-2), \\
&\pi_n^D s_i = s_{n-i} \pi_n^D & & (i=2,\dotsc,n-2), & 
&(\pi_1^D)^2=(\pi^D)^2=(\pi_n^D)^2=e \quad (i=1,\dotsc,n). 
\end{align}

Although it will not be used explicitly, 
let us write down the action of $W^D$ on $F_\bZ$ \eqref{eq:FbZ}.
\begin{align}
&s^D_0(\ep_i)=\begin{cases} c-\ep_2 & (i=1)   \\ c-\ep_1 & (i=2) \\
                              \ep_i & (i \neq 1,2)   \end{cases}, & 
&s_j  (\ep_i)=\begin{cases}   \ep_j & (i=j+1) \\ \ep_j+1 & (i=j) \\
                              \ep_i & (i \neq j, j+1)\end{cases} 
 \quad (j=1,\dotsc,n-1), \\
&s^D_n(\ep_i)=\begin{cases} -\ep_n & (i=n-1) \\-\ep_{n-1} & (i=n) \\
                             \ep_i & (i \neq n-1,n)  \end{cases}, &
&\pi_n^{D}(\ep_i) = \hf c-\ep_{n-i+1} \quad (i=0,\dotsc,n), \\
&\pi^D_1(\ep_i)=\begin{cases}c-\ep_1 & (i=1) \\ \ep_i & (i \neq 1) \end{cases}, & 
&\pi^D_{n-1}(\ep_i)=\begin{cases} \hf c+\ep_n & (i=1) \\
                                  \hf c-\ep_{n-i+1}   & (i \neq 1) \end{cases}.
\end{align}
We also write down reduced expressions of $t(\ep_i) \in W^D$:
\begin{align}\label{eq:tepD}
 t(\ep_1) = \pi_1^D, \quad 
 t(\ep_2) = \pi_1^D s_0^D s_1, \quad
 t(\ep_i) = \pi_1^D s_{i-1} \dotsm s_2 s_0^D s_1 \dotsm s_{i-1} \quad (i=3,\dotsc,n).
\end{align}

\subsubsection{Ram-Yip formula of non-symmetric Macdonald polynomial of type $D$}

There is a unique $W^D$-orbit on the affine root system $S^D$, i.e., $O_5$,
and correspondingly we set the parameter
\begin{align}
 t \lrto O_5.
\end{align}
See also \ref{eq:D:kN}. For $\mu \in P_{D_n}$, 
the non-symmetric Macdonald polynomial of type $D_n$ is denoted by 
\begin{align}
 E^D_\mu(x) = E^D_\mu(x; q,t).
\end{align}

For each $a=\alpha+r c\in S^D$, we define $\sh^D(a)$ and $\hgt^D(a)$ by
\begin{align}\label{eq:shhgtD}
 q^{ \sh^D(\alpha+r c)} := q^{-r}, \quad
 t^{\hgt^D(\alpha+r c)} := t^{\br{\rho^D,\alpha}}, \quad 
 \rho^D=\sum_{i=1}^n (n-i) \ep_i.
\end{align}
We also denote by $w_D(\mu) \in W^D$ the shortest element in the coset $t(\mu)W_0^D$.
For $\mu=\ep_i$, $i=1,2,\dotsc,n$, they are given by 
\begin{align}\label{eq:wmuD}
 w_D(\ep_1) = \pi^D_1, \quad 
 w_D(\ep_2) = \pi^D_1 s_0^D, \quad 
 w_D(\ep_i) = \pi^D_1 s_{i-1} \dotsm s_2 s_0^D \quad (3 \le i \le n).
\end{align}
The fundamental alcove of type $D_n$ is denoted by 
\begin{align} 
 A^D := \pr{x \in V \mid a_i^D(x) \ge 0, \ i=0,1,\dotsc,n}.
\end{align}
Finally, we denote by $\Gamma_D(\oa{w},z)$ the set of all alcove walks 
with start $z \in W^D$ of type $\oa{w}$,

\begin{fct}[{\cite[Theorem 3.1]{RY}}]\label{fct:DRY}
For $\mu \in P_{D_n}$, take a reduced expression 
$w_D(\mu)=\pi_j^D s_{i_1} \dotsm s_{i_r}$ of the element $w_D(\mu) \in W^D$
with some $j\in\pr{0,1,n-1,n}$. 
Then we have 
\begin{align}\label{eq:DRY}
&E^D_\mu(x)=\sum_{p\in\Gamma_D(\oa{w_(\mu)},e)} f^D_p
 t_{\dir(p)}^{\hf}x^{\wgt(p)}, \\
&f^D_p  := 
 \prod_{k\in\varphi_+(p)}(\psi^{D}_{i_k})^{+}(q^{\sh^D(-\beta_k)} t^{\hgt^D(-\beta_k))})
 \prod_{k\in\varphi_-(p)}(\psi^{D}_{i_k})^{-}(q^{\sh^D(-\beta_k))} t^{\hgt^D(-\beta_k)}),
\end{align}
where $\beta_k := s_{i_r} \dotsm s_{i_{k+1}}(a^D_{i_r})$ for $k=1,2,\dotsc,r$, 
and $(\psi^D_i)^{\pm}(z)$ for $i=0,1,\dotsc,n$ is given by 
\begin{align}
 &(\psi^D_i)^{\pm}(z) := \pm\frac{t^{-\hf}-t^{\hf}}{1-z^{\pm1}}.
\end{align}
\end{fct}

For distinction, we denote by $E_{\mu}^{D,\RY}(x;q,t)$ 
the right hand side of \eqref{eq:DRY}.

\subsubsection{Specialization to type $D_n$}

In this part, we specialize $t_n=u_n=t_0=u_0=1$ in $E_\mu(x)$ in Fact \ref{fct:RYOS}, 
and show that it is equal to $E_\mu^{D,\RY}(x)$ in Fact \ref{fct:DRY}. 
We denote the specialized Koornwinder polynomial by 
\begin{align}\label{eq:RY:spD}
 E_\mu^{\tsp,D}(x;q,t) := E_\mu(x;q,t,1,1,1,1).
\end{align}

Let $\Gamma_{0,n}(\oa{w(\mu)},e) \subset \Gamma(\oa{w(\mu)},e)$
be the subset consisting of alcove walks without folding by $s_0$ or $s_n$.

\begin{lem}\label{lem:SpDRY}
For any $\mu \in P_{C_n}$, take a reduced expression 
$w(\mu)=s_{i_1} s_{i_2} \dotsm s_{i_r}$ of the element $w(\mu) \in W$ in \eqref{eq:w(mu)}. 
Then we have
\begin{align}
&E^{\tsp,D}_\mu(x) = \sum_{p \in \Gamma_{0,n}(\oa{w(\mu)},e)} f_p
 t_{\dir(p)}^{\hf} x^{\wgt(p)}, \\
&f_p  := 
 \prod_{k\in\vp_+(p)}(\psi^{\tsp,D}_{i_k})^+(q^{\sh(-\beta_k)} t^{\hgt(-\beta_k)})
 \prod_{k\in\vp_-(p)}(\psi^{\tsp,D}_{i_k})^-(q^{\sh(-\beta_k)} t^{\hgt(-\beta_k)}), \\
&(\psi^{\tsp,D}_i)^{\pm}(z) :=
 \pm\frac{t^{-\hf}-t^{\hf}}{1-z^{\pm1}} \quad (i=1,2,\dotsc,n-1).
\end{align}
\end{lem}

\begin{proof}
The specialization $t_0=t_n=u_0=u_n=1$ in \eqref{eq:psi} yields
$\psi_0^{\pm}(z)=\psi^{\pm}_n(z)=0$. Thus the folding steps by $s_0$ or $s_n$ 
does not appear in the summation of Fact \ref{fct:RYOS}.
\end{proof}

Thus, it is enough to construct a bijection
$\Gamma_{0,n}(\oa{w(\mu)},e) \to \Gamma_D(\oa{w_D(\mu)},e)$.

\begin{lem}\label{lem:WDhom}
The following gives an injective group homomorphism $\vp^D: W \to W^D$.
\begin{align}
 \vp^D(s_0) = \pi_1^{D},\quad
 \vp^D(s_i) = s_i \quad (1 \le i \le n-1),\quad
 \vp^D(s_n) = e.
\end{align}
\end{lem}

\begin{proof}
We can check that the relations \eqref{eq:Wrel} of W are mapped by $\vp^D$ to 
those \eqref{eq:WDrel} of $W^D$.
Indeed, as for the final relation 
$s_0 s_1 s_0 s_1 = s_1 s_0 s_1 s_0$ in \eqref{eq:Wrel}, we have
\begin{align}
 \vp^D(s_0 s_1 s_0 s_1)  =     \vp^D(s_1 s_0 s_1 s_0) \iff
 \pi^D_1 s_1 \pi^D_1 s_1 = s_1 \pi^D_1 s_1 \pi^D_1 \iff
               s_0^D s_1 = s_1 s_0^D,
\end{align}
which is in the second line of \eqref{eq:WDrel}.
The other relations can be checked similarly. 
\end{proof}

\begin{lem}\label{lem:isomD}
For any $\mu\in P_{C_n}$, take a reduced expression 
$w(\mu)=s_{i_1} s_{i_2} \dotsm s_{i_{\ell}}$
of the element $w(\mu) \in W$ in \eqref{eq:w(mu)}, and set 
\begin{align}
&I_0 := \pr{ r \in \pr{1,2,\dotsc,\ell} \mid i_r \neq 0}, \quad 
 I_n := \pr{ r \in \pr{1,2,\dotsc,\ell} \mid i_r \neq n}, \\
&I := I_0 \cup I_n = \pr{k_1<k_2<\dotsb<k_s} \quad  (s \le \ell), \quad 
 J := \bpr{(b_1,b_2,\dotsc,b_\ell) \in \pr{0,1}^\ell \mid  b_i=1\ (i \notin I) }.
\end{align}
Using them, define $\theta^D\colon J \to \pr{0,1}^s$ by
\begin{align}
 J \ni (b_1,b_2,\dotsc,b_\ell) \lmto (b_{k_1},b_{k_2},\dotsc,b_{k_s}) \in \pr{0,1}^s.
\end{align}
\begin{enumerate}[nosep]
\item 
The length of $w_D(\mu) \in W^D$ is equal to $\abs{I}=s$, and 
\begin{align}
w_{D}(\mu)=
\begin{cases}
         s_{j_1} s_{j_2} \dotsm s_{j_s} & (\abs{I_0}    \in 2\bZ) \\
 \pi_1^D s_{j_1} s_{j_2} \dotsm s_{j_s} & (\abs{I_0} \notin 2\bZ)
\end{cases}.
\end{align}

\item
The map $\theta^D : J\rightarrow \pr{0,1}^s$ induces a bijection 
\begin{align}
&\wt{\theta^D}: \Gamma_{0,n}(\oa{w(\mu)},e) \to \Gamma_D(\oa{w_D(\mu)},e), \\
&(A,s_{i_1}^{b_1}A, \dotsc, s_{i_1}^{b_1} \dotsm s_{i_{\ell}}^{b_\ell}A) \lmto 
 \begin{cases}
 (A_D, s_{j_1}^{b_{k_1}}A_D,\dotsc,s_{j_1}^{b_{k_1}} \dotsm s_{j_s}^{b_{k_s}}A_D) 
 & (\abs{I_0} \in 2\bZ) \\
 (\pi_1^D A_D, \pi_1^D s_{j_1}^{b_{k_1}}A_D,\dotsc,
  \pi_1^D s_{j_1}^{b_{k_1}} \dotsm s_{j_s}^{b_{k_s}}A_D) 
 & (\abs{I_0} \notin 2\bZ)
\end{cases}.
\end{align}

\item
For any $p\in \Gamma_{0,n}(\oa{w(\mu)},e)$, we have 
$\wgt(p)=\wgt(\wt{\theta^D}(p))$, $\dir(p)=\dir(\wt{\theta^D}(p))$.
\end{enumerate}
\end{lem}

\begin{proof}
\begin{enumerate}
\item 
It is enough to show $\vp^D(w(\mu))=w_{D}(\mu)$ for any $\mu \in P_{C_n}$.
By the reduced expressions \eqref{eq:wmuCC} and \eqref{eq:wmuD}, we have 
$\vp^D(w(\ep_i))=w_D(\ep_i)$ for each $i=1,2,\dotsc,n$. 
Then, since $\vp^D$ is a group homomorphism by Lemma \ref{lem:WDhom}, 
we find the desired equality. 

\item
It is an immediate consequence of (1) and the bijectivity of $\wt{\theta^D}$.

\item
Similarly as (1), we have $\vp^D(t(\ep_i))=t(\ep_i)$ for each $i=1,2,\dotsc,n$,
and thus $t(\wgt(p))=t(\wgt(\wt{\theta^D}(p)))$ for each 
$p \in \Gamma_{0,n}(\oa{w(\mu)},e)$, which implies $\wgt(p)=\wgt(\wt{\theta^D}(p))$. 
As for the remaining $\vp^D(\dir(p))=\dir(\wt{\theta^D}(p))$, 
since $\vp^D(s_n)=e$ and $\vp^D$ preserves $s_1,s_2,\dotsc,s_{n-1}$,
the specialization $t_n=1$ yields $t_{\dir(p)}=t_{\dir(\wt{\theta^D}(p))}$,
which givers the consequence.
\end{enumerate}
\end{proof}

Thus we have $E_\mu^{\tsp,D}(x;q,t) = E_\mu^{D,\RY}(x;q,t)$ 
for any $\mu \in P_{C_n} \subset P_{D_n}$. 
Using \eqref{eq:RY:spD}, we have the conclusion:

\begin{prp}\label{prp:RY:D}
For any $\mu \in P_{C_n}$, the following equality holds.
\begin{align}
 E_\mu(x;q,t,1,1,1,1) = E_\mu^{D,\RY}(x;q,t).
\end{align}
\end{prp}

\section{Concluding remarks}\label{s:rmk}

The original motivation of our study on specialization is to find some explicit formula 
of symmetric Macdonald-Koornwinder polynomials, bearing in mind 
the Macdonald tableau formula \cite[Chap.\ VI, (7.13), (7.13')]{Mb} for type $\GL_n$.
Certain progress has been developed for such tableau formulas of type $B,C,D$ and 
$(C_n^\vee,C_n)$ by the recent papers \cite{FHNSS,HS1,HS2}, although the connection to 
Ram-Yip type formulas seems to be still unclear.

Another interesting theme is the $t= \infty$ limit.
By Sanderson \cite{San} and Ion \cite{I}, it is known that the graded character of 
the level one (thin) Demazure module of an affine Lie algebra of type $X_l^{(r)}$,
$X=A,D,E$, is equal to non-symmetric Macdonald polynomial of the corresponding type 
specialized at $t =  \infty$ if $X_l^{(r)} \neq A_{2l}^{(2)}$, and equal to
non-symmetric Koornwinder polynomial specialized at $t = \infty$ in $A_{2l}^{(2)}$.
There are vast amount of literature on this topic from representation-theoretic,
combinatoric, and geometric points of view.
For example, Orr and Shimozono \cite{OS} studied the relation of the limits
and quantum Bruhat graphs.
Let us also mention the article \cite{Chi} by Chihara, 
where the Demazure specialization for type $A_{2l}^{(2)}$ is derived by
the geometric consideration of a Demazure slice of the same type $A_{2l}^{(2)}$.

Returning to out study, it would be interesting to find a concrete connection 
between our argument and the geometric argument given in \cite{Chi}.
Let us close this note by a naive explanation on why 
the non-symmetric Koornwinder polynomial is related to the representation theory
of the affine Lie algebra of type $A_{2l}^{(2)}$.
According to \cite[\S 3.2]{I} and \cite[\S 1.5]{Chi}, one considers the specialization
of the Noumi parameters 
\begin{align}\label{eq:spI}
 (t,t_0,t_n,u_0,u_n)=(t,t,t,1,t).
\end{align}
Here we exchanged the specialized value of $(t_0,u_0)$ and $(t_n,u_n)$ in loc.\ cit., 
due to the numbering of roots explained below.
Comparing \eqref{eq:spI} and the specialization Table \ref{tab:sp}, 
we find that \eqref{eq:spI} is included as the case $t_m=t_l^2=t_s=t$ 
in the $BC_n$ specialization of \S \ref{ss:BC}:
\begin{center}
\begin{tabular}{l||lllll}
       & $t$   & $t_0$   & $t_n$  & $u_0$ & $u_n$  \\ \hline
$BC_n$ & $t_m$ & $t_l^2$ & $t_s$  & $1$   & $t_s$   
\end{tabular}
\end{center}
Let us write again the Dynkin diagram \eqref{dynkin:BC} 
of the affine root system of type $BC_n$:
\begin{align}
 \dynkin[mark=o, edge length=.75cm, reverse arrows, 
         labels={}, labels*={0,1,2,,,n-1,n}] A[2]{even}
\end{align}
This is in fact the Dynkin diagram for the affine Lie algebra of type $A^{(2)}_n$
for even $n$ \cite[p.55, \S 4.8, Table Aff 2]{Ka},
with the numbering of the roots $0,1,\dotsc,n$ reversed.
Thus, very naively speaking, we can read the result of Ion on the Koornwinder 
specialization \cite[\S 3]{I} from our specialization Table \ref{tab:sp}.

\begin{Ack}
The authors thank Professor Masatoshi Noumi for valuable comments.
In particular, Remarks \ref{rmk:C} and \ref{rmk:B} are added 
as a result of the correspondences with him.
\end{Ack}


\end{document}